\documentclass[11pt]{article}
\def\mydate{8 August  2016}

\usepackage{amsmath, amssymb}
\usepackage{amsfonts}
\usepackage{graphicx}
\usepackage{subfigure}
\usepackage{theorem}
\usepackage{enumerate}
\usepackage[dvips]{color}
\usepackage[normalem]{ulem}

\newcommand{\zH}{\mathcal H}

\newcommand{\zV}{\mathcal V}

\newcommand{\?}[1]{%
  \marginpar{%
    \begin{minipage}{2cm}
      \begin{flushleft}%
      \end{flushleft}%
   \end{minipage}%
  }%
}%

\newcount\claimno
\def\myclaim#1#2{
   \global\advance\claimno by 1\relax
   \bigskip\noindent\rlap{\rm(\the\claimno)}\ignorespaces
   \global\expandafter\edef\csname CLAIMLABEL#1\endcsname{(\the\claimno)}\relax
   \hangindent=33pt\hskip30pt{\sl#2}\bigskip}
\def\hardclaim#1#2{
   \bigskip\noindent\rlap{\rm(#1)}\ignorespaces
   \hangindent=33pt\hskip30pt{\sl#2}\bigskip}

\def\refclaim#1{\csname CLAIMLABEL#1\endcsname}
\def\myproof{\noindent{\bf Proof. }}
\def\qed{\hfill$\square$\bigskip\medskip}
\font\smallrm=cmr8
\def\rt#1{#1}
\def\yy#1{{\color{blue}#1}}
\def\zz#1{{\color{blue}#1}}
\def\xx#1{{\color{red}#1}}
\def\cc#1{{\color{red}#1}}
\def\cc#1{{#1}}
\def\zz#1{{#1}}
\def\xx#1{{#1}}
\def\yy#1{{\color{black}#1}}

\newtheorem{Proposition}{Proposition}[section]

\newtheorem{Lemma}[Proposition]{Lemma}

\newtheorem{Theorem}[Proposition]{Theorem}

\newenvironment{proof}%
{\noindent{\bf Proof.}\ }%
{\unskip\nobreak\hfill\penalty50\hskip2em\hbox{}\nobreak\hfill%
       $\square$\parfillskip=0pt\finalhyphendemerits=0\par\bigskip}%

{\theorembodyfont{\rmfamily}
\newtheorem{Definition}{Definition}
}

\usepackage[top=3cm, bottom= 3cm, left= 3cm, right= 3cm, asymmetric]{geometry}
\begin{document}
{
\baselineskip=16pt
\phantom{a}\vskip .25in
\centerline{{\bf  A NEW PROOF OF THE FLAT WALL THEOREM}}
\vskip.4in
\centerline{{\bf Ken-ichi Kawarabayashi}%
  \footnote{Supported by JST, ERATO, Kawarabayashi Large Graph Project.}
}
\centerline{National Institute of Informatics}
\centerline{2-1-2 Hitotsubashi, Chiyoda-ku, Tokyo 101-8430, Japan}
\medskip
\centerline{{\bf Robin Thomas}%
\footnote{\smallrm Partially supported by NSF under
Grants No.~DMS-0701077 and~DMS-1202640.}}
\centerline{School of Mathematics}
\centerline{Georgia Institute of Technology}
\centerline{Atlanta, Georgia  30332-0160, USA}
\medskip
\centerline{and}
\medskip
\centerline{{\bf Paul Wollan}%
\footnote{\smallrm 
Partially supported by the European Research Council under the European Union's Seventh Framework Programme (FP7/2007-2013)/ERC Grant Agreement no. 279558.}}
\centerline{Department of Computer Science}
\centerline{University of Rome ``La Sapienza"}
\centerline{00198 Rome, Italy}
}

\vfill \noindent 12 June 2012, revised \mydate. 
\newpage

\vskip 0.7in \centerline{\bf ABSTRACT}
\bigskip
We give an elementary and self-contained proof,
and a numerical improvement, of a weaker form
of the excluded clique minor theorem of Robertson and Seymour,
the following.
Let $t,r\ge1$ be  integers, and let $R=49152t^{24}(\rt{40}t^2+r)$.
An $r$-wall is obtained from a $2r\times r$-grid by deleting
every odd vertical edge in every odd row and every even vertical
edge in every even row, then deleting the two resulting vertices
of degree one, and finally subdividing edges arbitrarily.
The vertices of degree two that existed before the subdivision are
called the pegs of the $r$-wall.
Let $G$ be a graph with no $K_t$ minor, and let $W$ be an $R$-wall in $G$. 
We prove that there exist
a set $A\subseteq V(G)$ of size at most $12288t^{24}$ and an $r$-subwall
$W'$ of $W$ such that $V(W')\cap A=\emptyset$ and $W'$ is a flat wall
in $G-A$ in the following sense. There exists a separation $(X,Y)$ of $G-A$
such that $X\cap Y$ is a subset of the vertex set of the cycle $C'$
that bounds the outer face of $W'$, $V(W')\subseteq Y$, 
every peg of $W'$ belongs to $X$ and the graph
$G[Y]$ can almost be drawn in the unit disk with the vertices $X\cap Y$ drawn
on the boundary of the disk in the order determined by $C'$.
Here almost means that the assertion holds after repeatedly removing
parts of the graph separated from $X\cap Y$ by a cutset $Z$ of size at
most three, and adding all edges with both ends in $Z$.
Our proof gives rise to an algorithm that runs in polynomial time
even when $r$ and $t$ are part of the input instance.
The proof is self-contained in the sense that it uses only results
whose proofs can be found in textbooks.

\vfil\eject

\section{Introduction}

All graphs in this paper are finite, and may have loops and parallel edges.
A graph is a \emph{minor}\?{minor} of another if the first can be obtained
from a subgraph of the second by contracting edges.
An \emph{$H$ minor}\?{$H$ minor} is a minor isomorphic to $H$.
There is an ever-growing collection of so-called excluded minor theorems
in graph theory.
These are theorems which assert that every graph with no minor isomorphic
to a given graph or a set of graphs has a certain structure.
The best known such theorem is perhaps Wagner's reformulation
of Kuratowski's theorem~\cite{Wag37Kur}, 
which says that a graph has no $K_5$ or
$K_{3,3}$ minor if and only if it is planar.
One can also characterize graphs that exclude only one of those
minors.
To state such a characterization for excluded $K_5$ we need the following
definition.
Let $H_1$ and $H_2$ be graphs, and let $J_1$ and $J_2$ be complete subgraphs
of $H_1$ and $H_2$, respectively, with the same number of vertices.
Let $G$ be obtained from the disjoint union of $H_1-E(J_1)$ and 
$H_2-E(J_2)$ by choosing a bijection between $V(J_1)$ and $V(J_2)$
and identifying the corresponding pairs of vertices.
We say that $G$ is a \emph{clique-sum}\?{clique-sum} of $H_1$ and $H_2$.
Since we allow parallel edges, the set that results from the identification
of $V(J_1)$ and $V(J_2)$ may include edges of the clique-sum.
For instance, the graph obtained from $K_4$ by deleting an edge can be
expressed as a clique-sum of two smaller graphs, where one is a triangle and 
the other is a triangle with a parallel edge added.
By $V_8$ we mean the graph obtained from a cycle of length eight by
adding an edge joining every pair of vertices at distance four in the cycle.
The characterization of graphs with no $K_5$ minor, due to 
Wagner~\cite{Wag37}, reads as follows.

\begin{Theorem}
\label{thm:wagner}
A graph has no $K_5$ minor if and only if it can be obtained by
repeated clique-sums, starting from planar graphs and $V_8$.
\end{Theorem}

There are many other similar theorems; a survey can be found in~\cite{DieGrDec}.
Theorem~\ref{thm:wagner} is very elegant, but attempts at extending it
run into difficulties. For instance, no characterization is known for
graphs with no $K_6$ minor, and there is evidence suggesting that such
a characterization would be fairly complicated.
Even if a characterization of graphs with no $K_6$ is found,
there is no hope in finding one for excluding $K_t$ for larger values of $t$.

Thus when excluding an $H$ minor for a general graph $H$ we need to settle 
for a less ambitious goal---a theorem that gives a necessary condition 
for excluding an $H$ minor, but not necessarily a sufficient one.
However, for such a theorem to be meaningful, the structure it describes
must be sufficient to exclude some other, possibly larger graph $H'$.
For planar graphs $H$ this has been done by 
Robertson and Seymour~\cite{RS5}.
To state their theorem we need to recall that the 
\emph{tree-width}\?{tree-width} of a graph $G$ is the least integer
$k$ such that $G$ can be obtained by repeated clique-sums,
starting from graphs on at most $k+1$ vertices.

\begin{Theorem}
\label{thm:grid}
For every planar graph $H$ there exists an integer $k$ such that
every graph with no $H$ minor has tree-width at most $k$.
If $H$ is not planar, then no such integer exists.
\end{Theorem}

This is a very satisfying theorem, because it is best possible in
at least two respects. Not only is there no such integer when $H$ is not planar,
but no graph of tree-width $k$ has a minor isomorphic to the
$(k+1)\times(k+1)$-grid.

But how about excluding a non-planar graph?
Robertson and Seymour have an answer to that question as well,
but in order to motivate it we need to digress a bit.

\subsection{The Two Disjoint Paths Problem}

Let $C$ be a cycle in a graph $G$. We say
that a \emph{$C$-cross}\?{$C$-cross} in  $G$ is a pair of
disjoint paths $P_1,P_2$ with ends $s_1,t_1$ and $s_2,t_2$, respectively,
such that $s_1,s_2,t_1,t_2$ occur on $C$ in the order listed, and the
paths are otherwise disjoint from $C$.

Let $G$ be a graph, and let $s_1,s_2,t_1,t_2\in V(G)$.
The TWO DISJOINT PATHS PROBLEM asks whether there exist two disjoint
paths $P_1,P_2$ in $G$ such that $P_i$ has ends $s_i$ and $t_i$.
There is a beautiful characterization of the feasible instances,
which we now describe.
First of all, let us assume that $G$ has a cycle $C$ with vertex-set
$\{s_1,s_2,t_1,t_2\}$ in order.
This we can assume, because the edges of $C$ can be added without changing
the feasibility status of the problem.
It follows that the TWO DISJOINT PATHS PROBLEM is feasible if and only if
the graph $G$ has a $C$-cross.
Thus we will study the more general problem of when a graph has a $C$-cross.

Now if $G$ can be drawn in the plane with $C$ bounding a face, then
it has no $C$-cross.
(Proof. Add a new vertex in the face bounded by $C$ and join it by
an edge to every vertex of $C$. The new graph
is planar, and yet if the $C$-cross existed, they would give rise to a
$K_5$ minor in $G$.)
So this gives one class of obstructions, but there is another one.
A \emph{separation} in a graph $G$ is a pair $(A,B)$ of subsets of vertices such that
$A\cup B=V(G)$, and there is no edge of $G$ with one end in $A \setminus B$ and
the other in $B \setminus A$.
The order of the separation $(A,B)$ is $|A\cap B|$.
Now if there exists a separation $(A,B)$ of $G$ of order at most three with
$V(C)\subseteq A$, then the vertices in $B \setminus A$ are not very useful.
Let $H$ be the graph obtained from $G$ by deleting $B \setminus A$ and instead adding 
an edge joining every pair of vertices in $A\cap B$.
It follows that if $G$ has a $C$-cross, then so does $H$.
Furthermore, if we
choose $(A,B)$ so that some component of $G[B \setminus A]$ includes
a neighbor of every vertex in $A\cap B$, then the converse holds as well.
Let us turn this observation into a definition.


\begin{Definition} 
Let $G$ be a graph, and let $X\subseteq V(G)$.
Let $(A,B)$ be a separation of $G$ of order at most three with
$X\subseteq A$ and such that there exist $|A\cap B|$ 
paths from some vertex  $v\in B \setminus A$ to $X$ that are disjoint except for $v$.
Let $H$ be the graph obtained from $G[A]$ by adding an edge joining
every pair of \zz{distinct}  vertices in $A\cap B$.
We say that $H$ is an
\emph{elementary $X$-reduction}\?{elementary $X$-reduction} of $G$,
and we say that it is an \emph{elementary $X$-reduction determined by $(A,B)$}.
We say that a graph $J$ is an \emph{$X$-reduction}\?{$X$-reduction} of $G$
if it can be obtained from $G$ be a series of elementary $X$-reductions.
If $C$ is a subgraph of $G$, then by an (elementary) $C$-reduction
we mean an (elementary) $V(C)$-reduction.
\end{Definition}

Thus taking $C$-reductions does not change whether there exists a $C$-cross,
and as we are about to see, 
when no $C$-reduction is possible, the only obstruction to the 
existence of a $C$-cross is topological, namely that $G$
can be drawn in the plane with $C$ bounding a face.
The first version of the promised theorem, obtained in various forms by
Jung~\cite{Jung},
Robertson and Seymour~\cite{RS9},
Seymour~\cite{SeyDisj}, Shiloach~\cite{Shi}, and Thomassen~\cite{Tho2link}
reads as follows.

\begin{Theorem}
\label{thm:crossreduct}
Let $G$ be a graph, and let $C$ be a cycle in $G$.
Then $G$ has no $C$-cross if and only if some $C$-reduction of $G$
can be drawn in the plane with $C$ bounding a face.
\end{Theorem}

Since Theorem~\ref{thm:crossreduct} is not as well-known as it should be,
and its proof is not entirely trivial, we give a proof in the Appendix.
\yy{An additional reason for including a proof of Theorem~\ref{thm:crossreduct}
is to validate our claim that we only use results that can be found in textbooks.}
For applications it is desirable to have a representation of the
entire graph $G$ as opposed to some unspecified $C$-reduction.
Formalizing this idea is the subject to the next two definitions.

\begin{Definition}
\zz{%
If $X$ is a set in a topological space, we define $\widetilde X:=\overline X\setminus X$.
A {\em painting} in a surface $\Sigma$ is a pair $\Gamma=(U,N)$,
where $N\subseteq U\subseteq \Sigma$, $N$ is finite,
$U\setminus N$ has finitely many arcwise-connected components, called {\em cells},
and for every cell $c$, the closure $\overline c$ is a closed disk
and $\widetilde c=\overline c\cap N\subseteq\hbox{\rm bd}(\overline c)$
satisfies $|\widetilde c|\le3$.
We define $N(\Gamma):=N$, $U(\Gamma):=U$ and denote the set of cells
of $\Gamma$ by $C(\Gamma)$.
Thus the cells of a painting define a hypergraph with hyperedges of cardinality at most of three by saying that
$c$ is incident with the elements~$\widetilde c$.}
\end{Definition}

\begin{Definition}
\zz{%
Let $G$ be a graph, and let $\Omega$ be a cyclic permutation of a set $V(\Omega)\subseteq V(G)$.
By an {\em $\Omega$-rendition} of  $G$ we mean
a triple $(\Gamma,\sigma,\pi)$, where
\begin{itemize}
\item $\Gamma$ is painting in the unit disk $\Delta$,
\item $\sigma$ assigns to each cell $c\in C(\Gamma)$
a subgraph $\sigma(c)$ of $G$,  and
\item $\pi:N(\Gamma)\to V(G)$ is an injection
\end{itemize}
such that
\begin{itemize}
\item[(P1)]$G=\bigcup\left(\sigma(c)\,:\,c\in C(\Gamma)\right)$,
\item[(P2)]$\sigma(c)$ and $\sigma(c')$ are edge-disjoint
for distinct $c,c'\in C(\Gamma)$,
\item[(P3)]$\pi(\widetilde c)\subseteq V(\sigma(c))$ for every cell
$c\in C(\Gamma)$, 
\item[(P4)]$V(\sigma(c)\cap\bigcup\left(\sigma(c')\,:\,c'\in C(\Gamma)\setminus\{c\}\right))\subseteq \pi(\widetilde c)$
for every cell $c\in C(\Gamma)$, and 
\item[(P5)]the image under $\pi$ of $N(\Gamma)\cap \hbox{bd}(\Delta)$
is $V(\Omega)$, mapping the cyclic 
order of bd$(\Delta)$ to the cyclic order of $\Omega$.
\end{itemize}
A cycle $C$  defines a cyclic permutation of $V(C)$, and so we may speak of a {\em $C$-rendition}.}
\end{Definition}


Using the above definitions we can extend Theorem~\ref{thm:crossreduct}
as follows.

\begin{Theorem}
\label{thm:crossredrend}
Let $G$ be a graph, and let $C$ be a cycle in $G$.
Then the following conditions are equivalent:
\begin{enumerate}
\item[{\rm(1)}] $G$ has no $C$-cross,
\item[{\rm(2)}]  some $C$-reduction of $G$
can be drawn in the plane with $C$ bounding a face, and
\item[{\rm(3)}] \zz{$G$ has a $C$-rendition.}
\end{enumerate}
\end{Theorem}


\begin{proof}
\zz{%
The implication (1)$\Rightarrow$(2) holds by Theorem~\ref{thm:crossreduct}.
Let us now prove that (2)$\Rightarrow$(3)
 by induction on $|V(G)|$.
To that end 
let us assume that some $C$-reduction of $G$
can be drawn in the plane with $C$ bounding a face, and that the implication
(2)$\Rightarrow$(3) holds for all graphs on strictly fewer than $|V(G)|$ vertices.
We may assume that $G$ has no isolated vertices, because otherwise the 
implication follows by induction by deleting them.
Let us assume first that $G$ can be drawn in the plane with $C$ bounding a face.
We may assume that $V(C)$ is drawn on the boundary of the unit disk $\Delta$,
and that the rest of $G$ is drawn in the interior of $\Delta$.
We now construct a $C$-rendition as follows.
Let $F\subseteq E(G)$ be the set of all edges $e\in E(G)$ such that 
$e$ is not contained in the closed disk bounded by a loop
edge other than $e$, and let $V$ be the set of vertices $v\in V(G)$ that do not belong 
to the open disk bounded by a loop edge of $G$.
 For every edge $e\in F$ we
``fatten" $e$ into a disk $D_e$ in such a way that $e\subseteq D_e\subseteq\Delta$,
$D_e$ includes the two ends of $e$ in its boundary and is otherwise disjoint from $V$ and $F\setminus\{e\}$,
and  for distinct edges $e,e'\in F$ the intersection
$D_e\cap D_{e'}$ consists of common end(s) of $e$ and $e'$.
Let $\Gamma$ be the painting defined by $U(\Gamma)=\bigcup_{e\in F}D_e$
and $N(\Gamma)=V$. Thus each cell $c$ of $\Gamma$ includes an edge $e\in E(G)$
and we define $\sigma(c)$ to be the graph consisting of all vertices contained in $c$ 
and all edges contained in $c$ and their ends.
Thus  if $e$ is not  a loop, then it is the only edge contained in $c$.
Finally, we define $\pi:N(\Gamma)\to V(G)$ to be the identity.
Then $(\Gamma,\sigma,\pi)$ is a $C$-rendition of $G$, and hence (3) holds.
This completes the case when $G$ can be drawn in the plane with $C$ bounding a face.
}

\zz{%
We may assume now that some elementary $C$-reduction $G'$ of $G$ 
has a $C$-reduction that can be drawn in the plane with $C$ bounding a face.
Let $(A,B)$ be the separation of $G$  giving rise to the $C$-reduction~$G'$.
By the induction hypothesis $G'$ has a $C$-rendition $(\Gamma',\sigma',\pi')$.
If $A\cap B\subseteq\sigma'(c)$ for some $c\in C(\Gamma')$, then by adding
$G[B]$ to $\sigma'(c)$ we obtain a $C$-rendition of $G$, and so we may assume that
$A\cap B\not\subseteq\sigma'(c)$ for all $c\in C(\Gamma')$.
It follows that $|A\cap B|=3$, that $\pi(X)=A\cap B$ for some set $X\subseteq N(\Gamma)$
of size three, and that 
 every pair of elements of $X$ are incident
with a cell of $\Gamma'$. Thus there exists a closed disk $D$ whose boundary intersects $U(\Gamma')$
in $X$ only. Let the painting $\Gamma$ be defined by saying that $U(\Gamma)=U(\Gamma')\cup D$
and that $N(\Gamma)$ consists of all $n\in N(\Gamma')$ that do not belong to the interior of $D$.
Thus the cells of $\Gamma$ are $D\setminus X$
and all the cells of  $\Gamma'$ that are disjoint from $D$.
We define $\sigma(D\setminus X)$  to be the union of $G[B]$ and $\sigma'(c)$ over
all cells $c\in C(\Gamma')$ contained in $D$, and for cells  $c\in C(\Gamma')$  that are disjoint
from $D$ we define  $\sigma(c)= \sigma'(c)$.
Finally, let $\pi$ be the restriction of $\pi'$ to $N(\Gamma)$.
Then $(\Gamma,\sigma,\pi)$ is a $C$-rendition of~$G$.
This completes the proof of the implication (2)$\Rightarrow$(3).}

\zz{%
It remains to prove (3)$\Rightarrow$(1). We again proceed by induction on $|V(G)|$.
Let $(\Gamma,\sigma,\pi)$  be a $C$-rendition of $G$, and assume that the implication
(3)$\Rightarrow$(1) holds for all graphs on strictly fewer than $|V(G)|$ vertices.
Let us say that a cell $c\in C(\Gamma)$  is {\em slim} if $V(\sigma(c))\subseteq\pi(\widetilde c)$.
If every cell of $\Gamma$ is  slim, then it is easy to convert the $C$-rendition into
a drawing of $G$ in the plane with $C$ bounding a face, and hence $G$ has no $C$-cross, as desired.
We may therefore assume that there exists a cell $c\in C(\Gamma)$ that is not  slim.
Let $G'$ be obtained from $G$ by deleting $V(\sigma(c))-\pi(\widetilde c)$ and adding an edge
joining every pair of vertices in $\pi(\widetilde c)$.
It is easy to convert $(\Gamma,\sigma,\pi)$ to a  $C$-rendition of $G'$.
By induction the graph $G'$ has no $C$-cross, and it follows from the definition of $G'$
that neither does $G$, as desired.}
\end{proof}

\subsection{The Flat Wall Theorem}
\label{sec:weak}

We are now ready to formulate the weaker version of the excluded
$K_t$ theorem of Robertson and Seymour~\cite[Theorem~9.8]{RS13}.
Let us begin by describing it informally.
We use $[r]$ to denote $\{1,2,\ldots,r\}$.
Let  $r,s\ge2$ be  integers.
An $r\times s$-grid is the graph with vertex-set $[r]\times[s]$ in which
$(i,j)$ is adjacent to $(i',j')$ if and only if $|i-i'|+|j-j'|=1$.
An \emph{elementary $r$-wall}\?{elementary $r$-wall} is obtained from the  
$2r\times r$-grid
by deleting all edges with ends $(2i-1,2j-1)$ and $(2i-1,2j)$
for all $i=1,2,\ldots,r$ and $j=1,2,\ldots,\lfloor r/2\rfloor$
and all edges with ends  $(2i,2j)$ and $(2i,2j+1)$
for all $i=1,2,\ldots,r$ and $j=1,2,\ldots,\lfloor (r-1)/2\rfloor$
and then deleting the two resulting vertices of degree one.
An \emph{$r$-wall}\?{$r$-wall} 
is any graph obtained from an elementary $r$-wall by subdividing
edges.
In other words each edge of the elementary $r$-wall is replaced by 
a path. 
Figure~\ref{fig:wall} shows an elementary $4$-wall.
Walls are harder to describe than grids, but they are easier to work with;  moreover, 
if a graph has a $2r\times2r$-grid  minor, then it has a subgraph
isomorphic to an $r$-wall.
Let $W$ be an $r$-wall, where $W$ is a subdivision of an elementary
wall $Z$. Let $X$ be the set of vertices of $W$ that
correspond to vertices $(i,j)$ of $Z$ with $j=1$, and let $Y$ be the
set of vertices of $W$ that correspond to vertices $(i,j)$ of $Z$ with 
$j=r$.
There is a unique set of
$r$ disjoint paths $Q_1,Q_2,\ldots,Q_r$ in $W$, such that each has
one end in $X$ and one end in $Y$, and no other vertex in $X\cup Y$.
 We may assume that the
paths are numbered so that the first coordinates of their vertices
are increasing.
We say that $Q_1,Q_2,\ldots,Q_r$ are the 
\emph{vertical paths}\?{vertical paths} of $W$.
There is a unique set of $r$ disjoint paths with one end in $Q_1$,
the other end in $Q_r$, and otherwise disjoint from $Q_1\cup Q_r$.
Those will be called the \emph{horizontal paths}\?{horizontal paths} of $W$.
Let $P_1,P_2,\ldots,P_r$ be the horizontal paths numbered in the 
order of increasing second coordinates.
Then $P_1\cup Q_1\cup P_r\cup Q_r$ is a cycle, and we will call it the
\emph{outer cycle}\?{outercycle} of $W$.
If $W$ is drawn as a plane graph in the obvious way, then this is
indeed the cycle bounding the outer face.
The sets $V(P_1\cap Q_1)$, $V(P_1\cap Q_r)$, $V(P_r\cap Q_1)$, and
$V(P_r\cap Q_r)$ each include exactly one vertex of $W$;
those vertices will be called the \emph{corners}\?{corners} of $W$.
In Figure~\ref{fig:wall} the four corners are circled.
The vertices of $W$ that correspond to vertices of $Z$ of degree two
will be called the \emph{pegs} of $W$.
Thus given $W$ as a graph the corners and pegs are not necessarily uniquely
determined.
Finally let $W,W'$ be walls such that $W'$ is a subgraph of $W$.
We say that $W'$ is a \emph{subwall}\?{subwall} of $W$ if every
horizontal path of $W'$ is a subpath of a horizontal path of $W$,
and every vertical path of $W'$ is a subpath of a vertical path of $W$.

\begin{figure}[htb]
 \centering
\includegraphics[scale = .75]{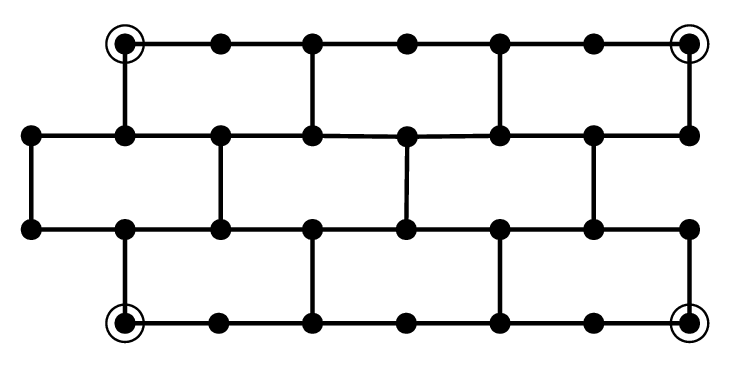}
 \caption{An elementary $4$-wall.}
\label{fig:wall}
\end{figure}



Now let $W$ be a large wall in a graph $G$ with no $K_t$ minor.
The Flat Wall Theorem asserts that there exist a set of vertices
$A\subseteq V(G)$ of bounded size and a reasonably big subwall $W'$
of $W$ that is disjoint from $A$ and has the following property.
Let $C'$ be the outer cycle of $W'$.
The property we want is that $C'$ separates the graph $G-A$ into two graphs,
and the one containing $W'$, say $H$, can be drawn in the plane with
$C'$ bounding a face.
However, as the discussion of the previous subsection attempted to
explain, the latter condition is too strong.
The most we can hope for is for the graph $H$ to be $C'$-flat.
That is, in spirit, what the theorem will guarantee,
except that we cannot guarantee that all of $C'$ be part of a planar
$C'$-reduction of $H$.
The correct compromise is that some subset of $V(C')$ separates off
the wall $W'$, and it is that subset that is required to be incident with
one face of the planar drawing.
Here is the formal definition.

\begin{Definition}
Let $G$ be a graph, 
and let $W$ be a wall in
$G$ with outer cycle $D$.
Let us assume that there exists a separation $(A,B)$ such that
$A\cap B\subseteq V(D)$,  $V(W)\subseteq B$,
and there is a choice of pegs of $W$ such that every peg belongs to $A$.
If some $A\cap B$-reduction of $G[B]$ can be drawn in a disk with
the vertices of $A\cap B$ drawn on the boundary of the disk in the
order determined by $D$, then we say that
the wall $W$ is \emph{flat in $G$}\?{flat in $G$}.
It follows that it is possible to choose the corners of $W$ is such
a way that every corner belongs to $A$.
\end{Definition}

We need one more definition. Given a wall $W$ in a graph $G$ we will 
(sometimes) produce a $K_t$ minor in $G$. 
However, this $K_t$ will not be arbitrary; it will be
very closely related to the wall $W$.
To make this notion precise we first notice that a $K_t$ minor in $G$
is determined by $t$ pairwise disjoint sets $X_1,X_2,\ldots,X_t$
such that each induces a connected subgraph and every two of the sets
are connected by an edge of $G$.
We say that $X_1,X_2,\ldots,X_t$ form a \emph{model}\?{model} of a $K_t$ 
minor and we will
refer to the sets $X_i$ as the \emph{branch-sets} of the model.
Often we will shorten this to a model of $K_t$.
\yy{%
Let $P_1,P_2,\ldots,P_r$ be the horizontal paths and 
$Q_1,Q_2,\ldots,Q_r$  the vertical paths of $W$.
We say that a model of a $K_t$ minor in $G$ is \emph{grasped}\?{grasped}
by the  wall $W$ if for every branch-set  $X_k$ of the model
there exist distinct indices $i_1,i_2,\ldots,i_t\in\{1,2,\ldots,r\}$ and distinct 
indices $j_1,j_2,\ldots,j_t\in\{1,2,\ldots,r\}$ such that $V(P_{i_l}\cap Q_{j_l})\subseteq X_k$
for all $l=1,2,\ldots,t$.
}%
Let us remark, for those familiar with the literature, that if 
a wall grasps a model of $K_t$, then 
\yy{the tangle determined by $W$}
controls it in the sense
of~\cite{RS16}.
The notion of control is important in applications, but since the stronger property is a consequence of the proof, 
we state the theorem that way.

We can now formulate the Flat Wall Theorem.
It first appeared in a slightly weaker form in~\cite[Theorem~9.8]{RS13} 
with an unspecified bound
on $R$ in terms of $t$ and $r$.

\begin{Theorem}
\label{thm:main}
Let $r,t \ge 1$ be integers, 
let $R=49152t^{24}(\rt{40}t^2+r)$, let $G$ be a graph, and
let $W$ be an $R$-wall in $G$.
Then either $G$ has a model of a $K_t$ minor grasped by $W$, or there exist
a set $A\subseteq V(G)$ of size at most $12288t^{24}$ and an $r$-subwall
$W'$ of $W$ such that $V(W')\cap A=\emptyset$ and $W'$ is a flat wall
in $G-A$.
\end{Theorem}

\yy{%
We can use Theorem~\ref{thm:main} to obtain an approximate  characterization 
of graphs with no large clique minor, as follows.
}

\yy{%
\begin{Theorem}
\label{thm:approxchar}
Let $r,t \ge 1$ be integers, 
let $R=49152t^{24}(\rt{40}t^2+r)$, and let $G$ be a graph.
If $G$ has no $K_t$ minor, then for every  $R$-wall $W$ in $G$ there exist
a set $A\subseteq V(G)$ of size at most $12288t^{24}$ and an $r$-subwall
$W'$ of $W$ such that $V(W')\cap A=\emptyset$ and $W'$ is a flat wall in $G-A$.
Conversely, if $t\ge2$ and $r\ge80t^{12}$ and for every  $R$-wall $W$ in $G$ there exist
a set $A\subseteq V(G)$ of size at most $12288t^{24}$ and an $r$-subwall
$W'$ of $W$ such that $V(W')\cap A=\emptyset$ and $W'$ is a flat wall in $G-A$,
then $G$ has no $K_{t'}$ minor, where $t'=2R^2$.
\end{Theorem}
}

\yy{%
\begin{proof}
The first part of the theorem follows immediately from Theorem~\ref{thm:main}.
To prove the converse suppose for a contradiction that 
$r\ge123t^{12}$ and that $G$ has a $K_{t'}$ minor, and yet for every  $R$-wall $W$ in $G$ there exist
a set $A\subseteq V(G)$ of size at most $12288t^{24}$ and an $r$-subwall
$W'$ of $W$ such that $V(W')\cap A=\emptyset$ and $W'$ is a flat wall in $G-A$.
Let $W_0$ be the elementary $R$-wall.
We may assume that $G$ has a $K_{t'}$ minor with model $(X_v\,:\,v\in V(W_0))$.
Since $W_0$ has maximum degree at most three, it follows that there exists a 
a subgraph $W$ of $G$ isomorphic to a subdivision of $W_0$  such that
\begin{itemize}
\item for every vertex $v\in V(W_0)$  the corresponding vertex of $W$ belongs to $X_v$, and
\item for every edge $uv\in E(W_0)$ the vertex-set of the corresponding path of $W$ 
is contained in $X_u\cup X_v$.
\end{itemize}
Thus $W$ is an $R$-wall in $G$, and hence  there exist
a set $A\subseteq V(G)$ of size at most $12288t^{24}$ and an $r$-subwall
$W'$ of $W$ such that $V(W')\cap A=\emptyset$ and $W'$ is a flat wall in $G-A$.
Let $W_0'$ be the elementary subwall of $W_0$ that corresponds to $W'$.
Since $t\ge2$ and $r\ge80t^{12}$  there exist five distinct vertices $v_1,v_2, \dots,v_5\in V(W_0')$
such that none of them belongs to the outer cycle of $W_0'$ and
$X_{v_i}\cap A=\emptyset$ for all $i=1,2,\ldots,5$.
Let $(X,Y)$ be a separation of $G-A$ as in the definition of a flat wall.
Since $X\cap Y$ is a subset of the vertex-set of the outer  cycle of $W'$ we deduce
that $X_i\cap X\cap Y=\emptyset$, and hence $X_i\subseteq Y$ for all $i=1,2,\ldots,5$.
Thus $X_{v_1},X_{v_2},\ldots,X_{v_5}$ is a model of a $K_5$ minor in $G[Y]$.
Furthermore, by considering the vertices $v_1,v_2,v_3,v_4,v_5$ of the wall $W'$
we conclude that there exist four disjoint paths in $W'$, and hence in $G[Y]$, such that the $i$-th  path has
one end in $X_{v_{j_i}}$ and the other end a peg of $W'$, and none of the paths has an internal 
vertex in any of the sets~$X_{v_j}$, where $j_1,j_2,j_3,j_4$ are pairwise distinct.
However, the existence of the $K_5$ minor and the four disjoint paths contradict the fact
that some $X\cap Y$-reduction of $G[Y]$ can be drawn in a disk.
\end{proof}
}

An earlier version of this paper as well as other articles refer to Theorem~\ref{thm:main}
as the Weak Structure Theorem. However, we prefer the current name, because it gives a more
accurate description of the result.

By  Theorem~\ref{thm:grid}
every graph of sufficiently large tree-width has an $R$-wall.
It follows from~\cite{kk} that in Theorem~\ref{thm:main}
the hypothesis that $G$ have an $R$-wall can be replaced by the assumption
that $G$ have tree-width at least $t^{\Omega(t^2\log t)}R$,
and by~\cite{ChuImproved} it can be replaced by the assumption 
that $G$ have tree-width at least $\Omega(R^{19}\hbox{ poly }\log R)$.

We prove Theorem~\ref{thm:main} in Section~\ref{sec:proof}.
Our proof is self-contained, but it is inspired by the Graph Minors
series of Robertson and Seymour.
\rt{%
Giannopoulou and Thilikos~\cite{GiaThi} improved the bound on the size of $A$ to the best
possible bound of $|A|\le t-5$.
Their proof uses Theorem~\ref{thm:gm16}, and therefore does not give an explicit bound
on $R$ as a function of $t$.
In Section~\ref{sec:fewapex} we deduce the bound of $|A|\le t-5$ from Theorem~\ref{thm:main}
by an elementary argument with an explicit bound on $R$, as follows.
}
 


\begin{Theorem}
\label{thm:mainvar}
Let $t \ge 5$ and $r \ge 3\lceil\sqrt t\rceil$ be integers.
Let $n=12288t^{24}$, $R=r^{2^n}$ and
$R_0=\yy{49152t^{25}(\rt{40}t+R)}$.
Let $G$ be a graph, and
let $W_0$ be an $R_0$-wall in $G$.
Then either $G$ has a model of a $K_t$ minor grasped by $W_0$, or there exist
a set $A\subseteq V(G)$ of size at most $t-5$ and an $r$-subwall
$W$ of $W_0$ such that $V(W)\cap A=\emptyset$ and $W$ is a flat wall
in $G-A$.
\end{Theorem}

In fact, in Theorem~\ref{thm:mainvar2} we prove a stronger result asserting that the set $A$ and subwall $W$ may be chosen
in such a way that every vertex of $A$ attaches throughout the wall $W$.
In Section~\ref{sec:fewapex} we prove another variation, where in the second outcome
we are able to conclude that if $(X,Y)$ is a separation that witnesses that $W$ is a flat wall,
then $G[Y]$ has bounded tree-width (or, equivalently, has no big wall).
That conclusion is useful in algorithmic applications, but in order to obtain it we need to
drop the conditions that the $K_t$ minor is grasped by the wall $W_0$ and that
the desired wall $W$ is a subwall of $W_0$. 

\begin{Theorem}
\label{thm:mainvar3}
Let $r\ge2$ and $t \ge 5$ and 
be integers,
let $n=12288t^{24}$ and
 $R_0=49152t^{24}(\rt{40}t^2+(rt)^{2^n})$ and let $G$ be a graph
with no $K_t$ minor.
If $G$ has an $R_0$-wall,
then there exist
a set $A\subseteq V(G)$ of size at most $t-5$ and an $r$-wall
$W$ in $G$ such that $V(W)\cap A=\emptyset$ and $W$ is a flat wall
in $G-A$.
Furthermore, if $(X,Y)$ is a separation as in the definition of flat wall,
then the graph $G[Y]$ has no $(R_0+1)$-wall.
\end{Theorem}

In Section~\ref{sec:algo} we convert the proof of Theorem~\ref{thm:main}
into a polynomial-time algorithm, as follows.

\begin{Theorem}
\label{alg:main}
There is an algorithm with the following specifications.\\
{\bf Input:} A graph $G$ on $n$ vertices and $m$ edges,
integers $r,t\ge1$,
and an $R$-wall $W$ in $G$, where
$R=49152t^{24}(\rt{60}t^2+r)$.\\
{\bf Output:}
Either a model of a $K_t$ minor in $G$ grasped by $W$, or
a set $A\subseteq V(G)$ of size at most $12288t^{24}$ and an $r$-subwall
$W'$ of $W$ such that $V(W')\cap A=\emptyset$ and $W'$ is a flat wall
in $G-A$.\\
{\bf Running time:} $O(t^{24}m+\yy{n})$.
\end{Theorem}

In the second alternative the algorithm also returns a separation $(A,B)$
as in the definition of flat wall, and a certificate that the
separation is as desired.
The details are in the version stated as Theorem~\ref{alg:wst}.

\subsection{The Excluded Clique Minor Theorem}

Theorem~\ref{thm:main} is a step toward a more
comprehensive excluded minor theorem of 
Robertson and Seymour~\cite{RS16}.

\begin{Theorem}
\label{thm:gm16}
For every finite graph $H$ there exists an integer $k$ such that
every graph with no $H$ minor can be obtained by repeated clique-sums,
starting from graphs that $k$-near embed in a surface in which $H$
cannot be embedded.
\end{Theorem}

Since we do not need Theorem~\ref{thm:gm16}, 
let us omit the precise definition of $k$-near embedding.
Instead, let us describe it informally.
A graph $G$ can be $k$-near embedded in a surface $\Sigma$ if there
exists a set $A\subseteq V(G)$ of size at most $k$ such that $G-A$
can be almost drawn in $\Sigma$, except for at most $k$ areas of
non-planarity, where crossings are permitted, but the graph is
restricted in a different way.
Here almost (similarly as in the abstract) means that we are not drawing
the graph $G$ itself, but some $C$-reduction instead, where
now $C$ is a large wall in $G$.
We refer to~\cite{RS16} for a precise statement.

We believe that we have found a much simpler proof of Theorem~\ref{thm:gm16}
with a significantly improved bound on $k$.
We will report on it soon.

The paper is organized as follows.
In the next three sections we prove auxiliary lemmas,
and in Section~\ref{sec:proof} we prove Theorem~\ref{thm:main}.
In Section~\ref{sec:fewapex} we prove Theorems~\ref{thm:mainvar}
\rt{and~\ref{thm:mainvar3}.}
In Section~\ref{sec:algo} we convert the proof of Theorem~\ref{thm:main}
to a polynomial-time
algorithm to construct either a $K_t$ minor or a flat wall.
In order to keep the paper self-contained 
we give a proof of Theorem~\ref{thm:crossreduct}
in the Appendix. 

\section{Disjoint $M$-paths with distance constraints}
\label{sec:mpaths}
%

Let $G$ be a graph, and let $M$ be a subgraph of $G$. 
By an {\em $M$-path} we mean
a path in $G$ with at least one edge, both ends in $V(M)$ and otherwise
disjoint from $M$.
The objective of this section is to study $M$-paths that are ``long"
in the sense that their ends are at least some specified distance
apart according to a metric on $V(M)$.
We prove an Erd\H{o}s-P\'osa-type result that says that either there
are many long $M$-paths, or all long $M$-paths can be destroyed by
deleting a restricted set of vertices.
In fact, we prove two closely related results along the same lines.
It turns out that for these lemmas the distance need not be given by 
a metric---all that is needed is the knowledge of which pairs of
vertices are far apart.
We capture that using the relation $R$  below.

\begin{Definition}
Let $G$ be a graph, let $M$ be a subgraph of $G$, and let $R$ be a reflexive
and symmetric relation on $V(M)$.
We say that pairwise disjoint $M$-paths $P_1, \dots, P_k$ 
are \emph{$R$-semi-dispersed\?{$R$-semi-dispersed}} 
if it is possible to label the 
ends of $P_i$ as $x_i$ and $y_i$ such that
 $(x_i, y_i)\not\in R$ and $(x_i, x_j) \not\in R$
for all distinct indices $i,j\in\{1,2,\ldots,k\}$.
Thus no restriction is placed on the relative position of the vertices
$y_1,y_2,\ldots,y_k$.
For $x\in V(M)$ we define $R(x)$, the {\em ball around $x$}, as
the set of all $y\in V(M)$ such that $(x,y)\in R$. 
\end{Definition}

Let us recall that a collection of paths $\cal P$ are
\emph{internally disjoint}\?{internally disjoint}
if every vertex that belongs to two distinct members of $\cal P$ is
an end of both.

\begin{Lemma}\label{lem:spreadpaths}
Let $G$ be a graph, let $M$ be a subgraph of $G$, let $R$ be a
reflexive and symmetric relation on $V(M)$, and let $k \ge 0$ be an
integer.  Then either there exist
pairwise disjoint $M$-paths $P_1, \dots, P_k$ which are $R$-semi-dispersed, 
or, alternatively, the following holds.  
There exist sets $A \subseteq V(G)$ and $Z \subseteq V(M)$ with $|A| \le k-1$
and $|Z| \le 3k-3$ such that every $M$-path $P$ in
$G-A$ with ends $x$ and $y$ either satisfies $(x, y)\in R$ or both
$x, y \in \bigcup_{z \in Z} R(z)$.   
\end{Lemma}


\begin{proof}
For the duration of the proof, we will say that an $M$-path $P$ is 
\emph{long} \?{long} if the ends $x$ and $y$ of $P$ satisfy $(x, y)\not\in R$.
Let $P_1, \dots, P_s$ be disjoint $M$-paths with the ends of $P_i$ labeled 
$x_i$ and $y_i$ satisfying the requirements in the definition of 
$R$-semi-dispersed.  
Let $0 \le p \le s$ be an integer, and let $Q_1, \dots, Q_p$ be disjoint 
paths with the ends of $Q_i$ equal to $a_i$ and $w_i$  satisfying the following
for all distinct integers $i,j\in\{1,2,\ldots,p\}$:
\begin{itemize}
\item[(a)] $w_i \in V(M) \setminus \bigcup_{m =1}^s (R(x_m) \cup R(y_m))$,
\item[(b)] $a_i \in V(P_i)$,
\item[(c)] $Q_i$ is internally disjoint from $V(M) \cup \bigcup_{m=1}^s V(P_m)$ and $E(Q_i) \cap E(M) = \emptyset$,
and
\item[(d)] $(w_i, w_j)\not\in R$.
\end{itemize}
We may assume that these paths are chosen so that $s$ is maximum, and,
subject to that, $p$ is maximum.
We may assume  that $s<k$, for otherwise the first
outcome of the lemma holds.
We will show that the sets
$ A:= \{a_1,a_2, \dots, a_p\}$ and 
$Z := \{w_1,w_2, \dots, w_{p}, x_1, y_1, \dots, x_{s}, y_s\}$ 
satisfy the second outcome of the lemma.


To that end 
let $W := \bigcup_{i=1}^p R(w_i)
\cup \bigcup_{i=1}^s \left( R(x_i) \cup R(y_i)\right)$.  
We may assume for a contradiction that there exists a long 
$M$-path $S$ in $G-A$ which has an end in $V(M)\setminus W$.  
If $S$ is disjoint from $P_1, \dots, P_s$,
we see that $S, P_1, \dots, P_s$ satisfy the definition of $R$-semi-dispersed,
contrary to the maximality of $s$.
Thus $S$ intersects one of the paths $P_i$, and hence we may
let $y$ be the first
vertex of $ \bigcup_{i=1}^p V(Q_i) \cup \bigcup_{i=1}^s V(P_i)$ which we
encounter when traversing the path $S$ beginning at an end in 
$x \in V(M)\setminus W$.

There are now several different cases, depending on where the vertex
$y$ lies. As the first case, assume $y \in V(Q_i)$ for some $1 \le i
\le p$.  It follows that $S \cup Q_i$ contains a long $M$-path, call
it $P'$, which has $x$ as an end and is disjoint from $P_1,
\dots, P_s$.  Then the paths $P', P_1, \dots, P_s$ are $R$-semi-dispersed, 
contrary to the maximality of $s$.  As the next case, assume
$y \in V(P_i)$ for some $1 \le i \le p$.  Then $S \cup Q_i \cup P_i$
contains two disjoint long $M$-paths, call them $P'$ and $P''$, such
that $P'$ has $x$ as an end and $P''$ has $w_i$ as an end.
Note that here we are using the property that $y \neq a_i$ to ensure
that $P'$ and $P''$ can be chosen disjoint.  Then the paths $P_1, \dots,
P_{i-1}, P_{i+1}, \dots, P_s, P', P''$ are $R$-semi-dispersed, again
contrary to the maximality of $s$.  As the final case,
consider when $y \in V(P_i)$ for some index $i$ with $p < i \le s$.  
We may assume, by swapping the paths $P_{p+1}$ and $P_i$, that $i=p+1$.
Then the paths $Q_1, \dots,Q_p,S$ contradict the maximality of $p$.

This completes the analysis of the possible cases, proving the lemma.  
\end{proof}

We also need the following closely related lemma.
Let $G$ be a graph, let $M$ be a subgraph of $G$, and let $R$ be a reflexive
and symmetric relation on $V(M)$.
We say that pairwise disjoint $M$-paths $P_1,P_2,\ldots,P_k$ are
\emph{$R$-dispersed}\?{$R$-dispersed} if $(x,y)\not\in R$
for every two distinct vertices $x,y$ such that each is an end of one
of the paths $P_i$.

\begin{Lemma}\label{lem:diversepaths}
Let $G$ be a graph, let $M$ be a subgraph of $G$,  let $R$ be a
reflexive and symmetric relation on $V(M)$, and let $k \ge 0$ be an
integer.  Then either there exist
pairwise disjoint $R$-dispersed $M$-paths $P_1, \dots, P_k$,
or, alternatively, the following holds.
There exist sets $A \subseteq V(G)$ and $Z \subseteq V(M)$ with $|A| \le k-1$
and $|Z| \le 3k-3$ such that for every $M$-path $P$ in
$G-A$ its  ends can be denoted by $x$ and $y$ such that either
$(x, y)\in R$ or 
$x \in \bigcup_{z \in Z} R(z)$.
\end{Lemma}

\begin{proof}
This follows by the same argument as Lemma~\ref{lem:spreadpaths}, 
with the following differences.
Instead of choosing the paths $P_i$ to be $R$-semi-dispersed we choose them
to be $R$-dispersed.
We choose the path $S$ to be an $(M- A-W)$-path in $G-A$;
if such a choice is not possible, then the lemma holds.
We then derive a contradiction as in the proof of Lemma~\ref{lem:spreadpaths}.
\end{proof}

\section{Meshes and clique minors}
\label{sec:mesh}

In this section we introduce the notion of a mesh---a common generalization
of walls and grids. It will allow us to reduce problems about walls
to problems about grids, which is useful, because grids are easier
to work with.
We also introduce a distance function on a mesh.

\begin{Definition}
Let $r, s \ge 2$ be positive integers, let $M$ be a graph, and let
$P_1,P_2,\ldots,P_r$, $Q_1,Q_2,\ldots,Q_s$ be paths in $M$ such that
the following conditions hold for all $i=1,2,\ldots,r$ and $j=1,2,\ldots,s$:
\begin{enumerate}
\item[(1)] $P_1,P_2, \dots, P_r$ are pairwise vertex disjoint, 
$Q_1,Q_2, \dots, Q_s$ are pairwise vertex disjoint, and
$M=P_1\cup P_2\cup \cdots\cup P_r\cup Q_1\cup Q_2\cup \cdots\cup Q_s$,
\item[(2)] $P_i\cap Q_j$ is a path, and if $i\in\{1,s\}$ or $j\in\{1,r\}$ or
both, then $P_i\cap Q_j$ has exactly one vertex,
\item[(3)] $P_i$ has one end in $Q_1$ and the other end in $Q_s$, and
when traversing $P_i$ the paths  $Q_1,Q_2,\ldots,Q_s$ are encountered
in the order listed,
\item[(4)] $Q_j$ has one end in $P_1$ and the other end in $P_r$,
and when traversing $Q_j$ the paths $P_1,P_2,\ldots,P_r$ are encountered
in the order listed.
\end{enumerate}
In those circumstances we say that $M$ is an 
\emph{$r \times s$ mesh}\?{$r \times s$ mesh}.
We will refer to $P_1,P_2,\ldots,P_r$ as {\em horizontal paths}
and to $Q_1,Q_2,\ldots,Q_s$ as {\em vertical paths}.
Thus every $r\times s$ grid is an $r\times s$ mesh, and
every planar graph obtained from an $r\times s$ grid by subdividing edges and
splitting vertices is an $r\times s$ mesh.
In particular, every $r$-wall is an $r\times r$-mesh.

We wish to define a distance function on a mesh, but we first do it for
a grid. Let $H$ be the $r\times s$ grid, so that $V(H)=[r]\times[s]$.
We regard $H$ as a plane graph, using the obvious straight-line drawing.
For $v_1=(x_1,y_1)$ and $v_2=(x_2,y_2)$ we define $d(v_1,v_2):=k-1$,
where $k$ is the least integer such that every curve in the plane joining
$v_1$ and $v_2$ intersects $H$ at least $k$ times. (We may clearly 
restrict ourselves to curves intersecting $H$ only in vertices.)
This distance can be calculated from the knowledge of the coordinates.
Indeed, it is easy to check that $d(v_1,v_2)$ is equal to the minimum
of $\max\{|x_1-x_2|,|y_1-y_2|\}$ and
$\min\{x_1,y_1,r+1-x_1,s+1-y_1\}+\min\{x_2,y_2,r+1-x_2,s+1-y_2\}-1$.

We now extend this definition to meshes as follows.
Let $M$ be a mesh with horizontal paths $P_1,P_2,\ldots,P_r$ and
vertical paths $Q_1,Q_2,\ldots,Q_s$ as above.
Then $M$ has an $H$ minor, where $H$ is the $r\times s$ grid, as in
the previous paragraph. 
Thus there exists a surjective mapping $f:V(M)\to V(H)$ such that 
$f^{-1}(u)$ is a branch-set of the $H$ minor for every $u\in V(H)$.
Furthermore, if $u=(i,j)$, then the set $f^{-1}(u)$ includes
$V(P_i)\cap V(Q_j)$.
If $d_H$ denotes the distance function on $H$ from the previous paragraph,
then we define $d(u,v):=d_H(f(u),f(v))$.
We say that $d$ is a {\em distance function on $M$}.
The function $d$ is a pseudometric; that is, it is symmetric and satisfies
the triangle inequality, but there may be distinct vertices
$u,v$ with $d(u,v)=0$. The function $d$ is not unique; it depends on the
choice of the function $f$.
\end{Definition}

\begin{Definition}
The definition of grasping extends to meshes almost verbatim, as follows.
\yy{%
Let $M$ be  an $r\times s$-mesh in a graph $G$ with 
 horizontal paths $P_1,P_2,\ldots,P_r$ and vertical paths $Q_1,Q_2,\ldots,Q_s$.
We say that a model of a $K_t$ minor in $G$ is \emph{grasped}\?{grasped}
by  $M$ if for every branch-set  $X$ of the model
there exist distinct indices $i_1,i_2,\ldots,i_t\in\{1,2,\ldots,r\}$ and distinct 
indices $j_1,j_2,\ldots,j_t\in\{1,2,\ldots,s\}$ such that $V(P_{i_l}\cap Q_{j_l})\subseteq X$
for all $l=1,2,\ldots,t$.
}%
\end{Definition}

Let $G$ be a graph and $M$ a mesh in $G$.  We first extend the definition of subwall to meshes in the natural way.

\begin{Definition} Let the horizontal and vertical paths of $M$ be $\zH$ and $\zV$, respectively. A mesh $M'$ with horizontal and vertical paths $\zH'$ and $\zV'$ is a \emph{submesh} of $M$ if every element of $\zH'$ is a subpath of a distinct element of $\zH$ and similarly, every element of $\zV'$ is a subpath of a distinct element of $\zV$.
\end{Definition}

\begin{Definition}  Let $G'$ be a minor of $G$ and $M'$ a mesh in $G'$.  We say that $M'$ is \emph{compatible} with
\zz{a mesh} $M$ \zz{in $G$} if there exist a subset $Z \subseteq E(G)$ and a submesh $\bar{M}$ of $M$ such that $G'$ is obtained from \rt{a subgraph of} $G$ by contracting $Z$ and $M'$ is obtained from $\bar{M}$ by contracting $Z \cap E(\bar{M})$.
\end{Definition}

\begin{Lemma}
\label{lem:controltrans}
Let $G$ be a graph, let $M$ be a mesh in $G$, let $G'$ be a minor of $G$,
and let $M'$ be a mesh in $G'$ compatible with $M$.
If for some integer $t\ge0$ the graph $M'$ grasps a $K_t$ minor of $G'$, 
then $M$ grasps a $K_t$ minor of $G$.
\end{Lemma}

The proof is clear and we omit it.

Let $r\ge1$ be an integer, and 
let $H_{2r}$ be the
 $2r\times 2r$-grid with vertex-set $[2r]\times[2r]$,
as usual.
The graph \emph{$H^1_{2r}$}\?{$H^1_{2r}$} is defined as the graph obtained from
$H_{2r}$ by adding all edges with ends $(i,r)$ and $(i+1,r+1)$,
and all edges with ends  $(i,r+1)$ and $(i+1,r)$ for all
$i=1,2,\ldots,2r-1$.
%
In other
words, $H^1_{2r}$ is constructed from the $2r\times 2r$-grid by
adding a pair of crossing edges in each face of the middle row of
faces.  
We will refer to the grid $H_{2r}$ as the \emph{underlying grid} of $H^1_{2r}$

\begin{Lemma}\label{lem:H1}
Let $t \ge 2$ be an integer.  The graph $H^1_{t(t-1)}$ has a $K_t$  minor
grasped by the underlying grid.
\end{Lemma}

\begin{proof}
The proof is by induction on $t$.  Let the vertices of $H^1_{t(t-1)}$ be 
labeled as in the definition, and let $L$ be the set of vertices of 
$H^1_{t(t-1)}$ with the second coordinate one.
We actually prove a slightly stronger statement, to facilitate the induction.  
We show that $H^1_{t(t-1)}$ has a $K_t$ minor grasped by the 
underlying grid such that every branch set contains a vertex in $L$.
The statement clearly holds for $t=2$, and so we assume that $t>2$ and that
the statement holds for $t-1$.

Let $H'$ be the subgraph of $H^1_{t(t-1)}$ induced by vertices $(x,y)$, where
$1\le x\le (t-1)(t-2)$ and $t\le y\le (t-1)^2$, and let $L'$ be the
set of vertices of $H'$ with second coordinate $(t-1)^2$.
Then $H^1_{(t-1)(t-2)}$ is isomorphic to $H'$ by an isomorphism that maps the
first row of $H^1_{(t-1)(t-2)}$ onto $L'$.
By the induction hypothesis the graph $H'$ has a $K_{t-1}$ minor 
with branch sets $X_1',X_2',\ldots,X'_{t-1}$ such that $X'_i\cap L'\not=\emptyset$
for all $i=1,2,\ldots,t-1$.
Let $i\in\{1,2,\ldots,t-1\}$. 
Let $x_i$ be such that $(x_i,(t-1)^2)\in X'_i\cap L'$.
We may assume that $x_1>x_2>\cdots>x_{t-1}$.
We define $X_i$ to consist of $X'_i$, the vertices 
$(x_i,(t-1)^2+i)$, $((t-1)(t-2)+2i-1,(t-1)^2+i)$,
$((t-1)(t-2)+2i-1,t(t-1)/2+1)$, 
$((t-1)(t-2)+2i,t(t-1)/2)$, 
$((t-1)(t-2)+2i,1)$, 
and the vertices of vertical and horizontal paths of the underlying grid
connecting those vertices, making each $X_i$ induce a connected subgraph
of $H^1_{t(t-1)}$.
Finally we define $X_t$ as the set containing all the vertices
$((t-1)(t-2)+2i-1,t(t-1)/2)$ and $((t-1)(t-2)+2i,t(t-1)/2+1)$
for all $i=1,2,\ldots, t-1$, and the vertices of the vertical path
connecting $((t-1)(t-2)+1,t(t-1)/2)$ to $((t-1)(t-2)+1,\yy{1})$.
\yy{This is illustrated in Figure~\ref{fig1}.
In order to satisfy the definition of grasping, we also add to $X_t$ the 
vertices $((t-1)(t-2)+1-i,t-i)$ and  $((t-1)(t-2)+2-i,t-i)$ for all $i=1,2,\ldots, t-1$
and the path joining them.}
It follows that $X_1,X_2,\ldots,X_t$ are the branch sets of a $K_t$ minor,
and each branch set intersects $L$.
\begin{figure}[htb]
 \centering
\includegraphics[scale = .75]{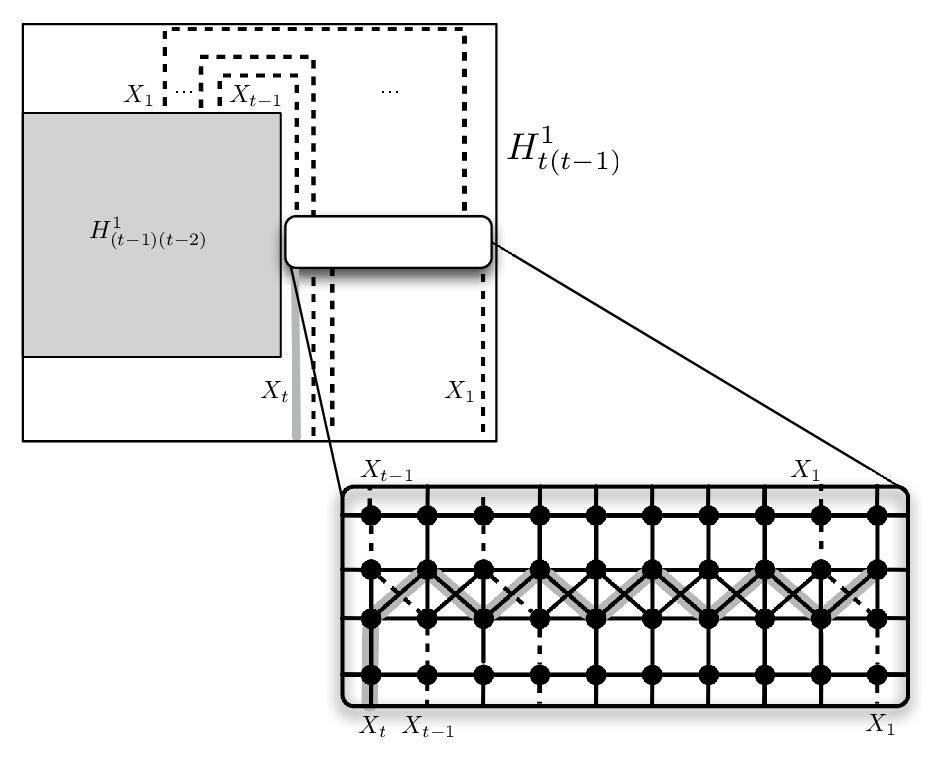}
 \caption{Finding a $K_t$ minor in $H^1_{t(t-1)}$.}
\label{fig1}
\end{figure}
\yy{%
Hence   each branch set $X_i$ satisfies the definition of grasping.
}%
We deduce that the minor is grasped by the underlying grid, 
as required.
\end{proof}

\section{Disjoint paths attaching to a mesh}

The goal of this section is to show that given a mesh $M$ 
in a graph $G$, either $G$ has a $K_t$ minor grasped by $M$,
or there exist bounded number of vertices and bounded number of
balls in $M$ of bounded radius such that after deleting those
vertices and balls, every $M$-path has its ends close to each other.

We will need two classic lemmas going forward.
\begin{Lemma}\label{lem:epintervals}
Let $k, r,s \ge 1$ be integers, and  let 
$\cal I$ be a set of $k$ intervals on the real line.  
If $k \ge (r-1)(s-1)+1$, then either $\cal I$ has a subset of $r$ pairwise
disjoint intervals, or $\cal I$ has a subset of $s$ intervals that have
non-empty intersection.
\end{Lemma}

\begin{Lemma}[Erd\H{o}s and Szeker\'es]\label{lem:es}
Let $r,s\ge1$ be integers.
Every sequence of $k\ge(r-1)(s-1)+1$  real numbers has either 
a non-decreasing subsequence of length $r$,
or a non-increasing subsequence of length $s$.
\end{Lemma}

\begin{Lemma}\label{lem:grida}
Let $t \ge 2$ be a positive integer, let  
$k= 32(t(t-1))^6$, let $G$ be a graph, let $M$ be a mesh in $G$
with distance function $d$, let $X \subseteq V(M)$ with 
$|X| = 2k$ such that $d(x, y) \ge 2t(t-1)$ for all  $x, y \in X$,
and let $F\subseteq E(G)-E(M)$ be a matching of size $k$ with
vertex-set $X$.
Then the graph $G$ has a $K_t$ minor grasped by $M$.  
\end{Lemma}

\begin{proof} 
The definition of distance function involves a grid minor of $M$.
Let $H$ be a grid minor of $M$ that gives rise to the distance function $d$.  
Then $H$ is obtained from $M$ by 
contracting a set of edges.  
Let $G'$ be the minor obtained from $G$ by contracting the same set of edges.
Then $F$ gives rise to a matching $F'$ in $G'$ of size $k$.
Given the way we defined the distance function on a mesh, the ends
of the edges in $F'$ are pairwise at distance at least $2t(t-1)$ with 
respect to the distance function on $H$.
If $G'$ has a $K_t$ minor grasped by $H$, then
$G$ has a $K_t$ minor grasped by $M$ by Lemma~\ref{lem:controltrans}.
Thus it suffices to prove the lemma when $M$ is grid.

We therefore assume for the rest of the proof that $M$ is a grid.
Let the vertices of $M$ be labeled $(x, y)$ for $1 \le x \le s$, 
$1 \le y \le r$.  
We number the edges in $F$ as $e_1,e_2,\ldots$ and denote the
ends of $e_i$ by $(x_i,y_i)$ and $(u_i,v_i)$.
There is at most one edge of $F$ which has an end with distance at most 
$t(t-1)-1$ from a vertex of the outer cycle of $M$.
We discard such an edge from $F$ if it exists. 
The remaining edges $e_i$ therefore satisfy
\begin{itemize}
\item[(1)] if $(x,y)$ is an end of $e_i$, then $t(t-1)<x<s+1-t(t-1)$
and $t(t-1)<y<r+1-t(t-1)$.
\end{itemize}
%
We may temporarily assume that for every $i$ either $x_i<u_i$,
or $x_i=u_i$ and $y_i>v_i$.
By reducing $F$ to no less than half its original size we may assume
that either  $y_i\le v_i$ for all $i$,
or $y_i>v_i$ for all $i$.
In the former case it follows that $x_i<u_i$ for all $i$.
In the latter case we reverse the second coordinate and then swap the 
coordinates (formally we map each vertex $(x,y)$ to $(r+1-y,x)$)
and conclude that we may assume that for at least half the indices $i$
\begin{itemize}
\item[($*$)] $x_i<u_i$ and $y_i\le v_i$.
\end{itemize}
%
By restricting ourselves to a subset of $F$ of size $4(t(t-1))^3$
we may assume that either $x_i\ne x_j$ for all remaining pairs
of distinct edges $e_i,e_j$, or that $x_i=x_j$ for all such pairs.
In the latter case notice that $|y_i-y_j|\ge 2t(t-1)\ge4$,
because $(x_i,y_i)$ and $(x_j,y_j)$ are at distance at least $2t(t-1)$.
In the latter case we swap the coordinates one more time to 
arrive at a set $\{e_1,e_2,\ldots,e_l\}\subseteq F$ such that
for all distinct indices $i,j=1,2,\ldots,l$ condition (1) holds and
\begin{itemize}
\item[(2)] either $x_i < u_i$, or $x_i=u_i$ and $|x_i-x_j|\ge4$,
\item[(3)] $x_i \neq x_j$, and
\item[(4)] $l \ge 4(t(t-1))^3$.
\end{itemize}
We apply Lemma~\ref{lem:epintervals} 
to the set of intervals $\{[x_i, u_i]: 1\le i\le l\}$.  
We conclude that either there exists a set
$I \subseteq \{1,2,\ldots,l\}$ of size at least $t(t-1)$ such that 
the intervals $\{[x_i, u_i]: i \in I\}$ are pairwise disjoint,
or there exist a set $J \subseteq  \{1,2,\ldots,l\}$ 
of size at least $4(t(t-1))^2$
and an integer $z$ such that $x_i\le z\le u_i$ for all $i\in J$.

Assume first that $I$ exists.
We claim that the graph obtained from $M$ by adding 
the edges $\{e_i: i\in I\}$ has an $H^1_{t(t-1)}$  minor,
where the underlying grid of $H^1_{t(t-1)}$ is compatible with $M$.
To see this we use the first and last $t(t-1)$ vertical and horizontal
paths of $M$
(notice that by (1) for $i\in I$ no end of $e_i$ belongs to any of those
paths), 
and use the edges $e_i$ to obtain the crossings
in the middle row of faces.
The $i^{\rm th}$ crossing will use vertices $(x,y)$ with
$t(t-1)\le y\le r+1-t(t-1)$ and
$x_i\le x\le u_i$ if $x_i<u_i$ and
$x_i-1\le x\le x_i+1$ if $x_i=u_i$. 
Condition (2) guarantees that the crossings will be pairwise disjoint.
By Lemma~\ref{lem:H1} the graph $H^1_{t(t-1)}$ has a $K_t$ minor grasped
by the underlying grid of $H^1_{t(t-1)}$.
By Lemma~\ref{lem:controltrans} the graph $G$ has a $K_t$ minor
grasped by $M$, as desired.
This completes the case when $I$ exists.
  
We may therefore assume that $J$ and $z$ exist.
By renumbering the indices 
we may assume that $x_1 < x_2 < \dots < x_{4(t(t-1))^2} < z$ 
and $u_i \ge z$ for all $1 \le i \le 4(t(t-1))^2$.
Let $M_1$ be the  subgraph of $M$ induced by vertices $(x, y)$ with 
$1 \le x < z$ and $1 \le y \le r$, and let
$M_2$ be the subgraph of $M$ induced by vertices $(x, y)$ with 
$z \le x \le s$ and $1 \le y \le r$.  
We see that $(u_i, v_i) \in V(M_2)$ for all $1 \le i \le 4(t(t-1))^2$.  
Let $P$ be a path in $M_2$ covering the vertices of $M_2$.  
The edges $e_i$ for $1 \le i \le 4(t(t-1))^2$ each have one end 
in $P$ and one end in $V(M_1)$.  
By Lemma~\ref{lem:es} there exists a sequence 
$1\le i_1<i_2<\cdots<i_{2t(t-1)}$ such that the ends of
$e_{i_1},e_{i_2},\ldots,e_{i_{2t(t-1)}}$ occur on $P$ in the order listed.
For $j=1,2,\ldots,t(t-1)$ we make use of the edges $e_{i_{2j-1}},e_{i_{2j}}$
and the subpath of $P$ connecting the ends of $e_{i_{2j-1}}$ and $e_{i_{2j}}$
to construct an $M_1$-path with ends $x_{i_{2j-1}}$ and $x_{i_{2j}}$.
The paths just constructed are pairwise vertex-disjoint, and,
similarly as in the previous paragraph, can be used to
deduce that $G$ has an $H^1_{t(t-1)}$ minor,
where the underlying grid is compatible with $M_1$, and hence with $M$.
By Lemma~\ref{lem:H1} the graph $H^1_{t(t-1)}$ has a $K_t$ minor grasped
by the underlying grid of $H^1_{t(t-1)}$.
By Lemma~\ref{lem:controltrans} the graph $G$ has a $K_t$ minor
grasped by $M$, as desired.
\end{proof}

\begin{Lemma}
\label{lem:match}
Let $B$ be a connected graph of maximum degree at most four, 
and let $Y\subseteq V(B)$.
Then there exist at least $(|Y|-1)/4$ disjoint paths in $B$,
each with at least one edge and with both ends in $Y$.
\end{Lemma}

\begin{proof}
By a {\em leaf} of $B$ we mean a vertex of degree one.
We may assume for a contradiction that the conclusion does not
hold and, subject to that, $|E(B)|$ is minimum.
Then $B$ is a tree, every leaf belongs to $Y$,
$|Y|\ge6$
and (by contracting the incident edge we see that) 
the unique neighbor of every leaf belongs to $Y$.
Let $L$ be the set of leaves of $B$.
Since $|Y|\ge6$ the graph $B-L$ is a tree on at least two vertices, 
and therefore we may select a leaf $t$ of $B-L$.
Since $t$ has degree at most four and $|Y|\ge6$,
the vertex $t$ is adjacent to at most three leaves of $B$.
Let $B'$ be the graph obtained from $B$ by 
deleting $t$ and all leaves of $B$ adjacent to it, and let $Y':=Y\cap V(B')$.
By the minimality of $B$ there exist at least $(|Y'|-1)/4\ge(|Y|-1)/4-1$
disjoint paths in $B'$, each with at least one edge and both ends in $Y'$.
By adding the path with vertex-set 
$\{t,t'\}$, where $t'$ is a leaf of $B$ adjacent to $t$,
we obtain
a collection as required in the lemma, a contradiction.
\end{proof}
Before the next lemma, let us remark that $3\times2^{12}=12288$. 

\begin{Lemma}
\label{lem:semidisperse}
Let $t\ge1$ be an integer, let $k_0:=12288(t(t-1))^{12}$,
let $G$ be a graph,  let $M$ be a mesh in $G$ with distance function $d$, and
assume that $G$ has a set $\cal P$ of cardinality $k_0$
of pairwise disjoint $M$-paths with the property that 
  the ends of every path $P\in{\cal P}$ can be denoted
by $x(P)$ and $y(P)$ in such a way that $x(P)$ and $y(P)$ are at distance
at least $\rt{10}t(t-1)$ for every $P\in\cal P$, and $x(P)$ and $x(P')$ are at distance
at least $\rt{10}t(t-1)$ for every two distinct paths $P,P'\in {\cal P}$.
Then $G$ has a $K_t$ minor grasped by $M$.
\end{Lemma}

\myproof
Let us define a relation $R$ on $V(M)$ by saying that $(x,y)\in R$ if
$d(x,y)<2t(t-1)$. By Lemma~\ref{lem:diversepaths} applied to the relation
$R$, graph $M$ and integer $k=32(t(t-1))^{6}$ we deduce that one of the 
two outcomes holds.
If the first outcome holds, then $G$ has $K_t$ minor grasped by $M$
by Lemma~\ref{lem:grida}, and hence our lemma holds.
Thus we may assume that the second outcome of Lemma~\ref{lem:diversepaths}
holds, and hence
there exist sets  $A \subseteq V(G)$ and $Z \subseteq V(M)$ with
$|A| \le k-1$
and $|Z| \le 3k-3$ such that

\hardclaim{1}{for every $M$-path $P$ in
$G-A$ its  ends can be denoted by $x$ and $y$ such that either
$d(x, y)<2t(t-1)$ or $d(x,z)<2t(t-1)$ for some $z \in Z$.}

The set ${\cal P}$ has a subset of size at least
$k_0-k$ such that each member is disjoint from $A$.
By (1) there exists $z\in Z$ and a subset of the latter set of paths 
of size at least
$(k_0-k)/(3k-3)$ such that every member $P$ of
the latest set has the property that one of $x(P),y(P)$ is at distance
at most $2t(t-1)$ from $z$.
Let $B$ denote the subgraph of $M$ induced by  vertices of $M$ at distance 
at most $2t(t-1)$ from $z$.
Since the vertices $x(P)$ are pairwise at distance at least $\rt{10}t(t-1)$,
we deduce that $x(P)\in V(B)$ for at most one of those paths $P$. 
By omitting that path we obtain a set
 ${\cal P}'\subseteq{\cal P}$ of disjoint $M$-paths in $G-A$
with $y(P)\in V(B)$ for every $P\in {\cal P}'$ and such that ${\cal P}'$
has cardinality at least
$(k_0-k)/(3k-3)-1\ge 128(t(t-1))^6$. 


\def\junk#1{}
\junk{
Let $P_1,P_2,\ldots,P_r$ be the vertical paths of $M$, and  let
$Q_1,Q_2,\ldots,Q_s$ be the horizontal paths of $M$.
We define $\cal Q$ to be the set of vertical and horizontal paths of $M$ 
consisting of all horizontal paths $P$ such that $P\cap B$ is
a subgraph of the vertical paths of $M$, 
and all vertical paths $Q$ such that $Q\cap B$ is
a subgraph of the horizontal paths of $M$.
(If $z$ is at distance at least $2t(t-1)$ from 
$P_1\cup P_r\cup Q_1\cup Q_s$, then this is equivalent to
saying that $\cal Q$ is the set of vertical and horizontal paths of $M$ 
that are disjoint from $B$; otherwise we need this more complicated
definition.)
We define a submesh $M'$ consisting of subpaths of members of $\cal Q$ as follows.
Let $I,J$ be such that $\cal Q$ consists of $P_i$ and $Q_j$ for all $i\in I$
and all $j\in J$.}

Let $P_1,P_2,\ldots,P_r$ be the vertical paths of $M$, and  let
$Q_1,Q_2,\ldots,Q_s$ be the horizontal paths of $M$.
Let $H$ be a grid minor of $M$ that gave rise to the distance function $d$
on $M$, and let $f:V(M)\to V(H)$ be the corresponding surjection 
as in the definition of distance function.
We define $\cal Q$ to be the set of vertical and horizontal paths 
$P$ of $M$ such that $P$ is not a subgraph of $B$ and there is no
vertex $x$ of $P$ such that $f(x)$ and $f(z)$ are connected by
a curve that intersects $H$ at most $2t(t-1)$ times and does not use
the outer face of $H$.
(If $z$ is at distance at least $2t(t-1)$ from 
$P_1\cup P_r\cup Q_1\cup Q_s$, then this is equivalent to
saying that $\cal Q$ is the set of vertical and horizontal paths of $M$ 
that are disjoint from $B$; otherwise we need this more complicated
definition.)
We define a submesh $M'$ consisting of subpaths of members of $\cal Q$ as follows.
Let $I,J$ be such that $\cal Q$ consists of $P_i$ and $Q_j$ for all $i\in I$
and all $j\in J$.
Let $i_0:=\min I$, $i_1:=\max I$, $j_0:=\min J$, and $j_1:=\max J$.
For $i\in I$ let $P'_i$ be the shortest subpath of $P_i$ from $Q_{j_0}$
to $Q_{j_1}$, and 
for $j\in J$ let $Q'_j$ be the shortest subpath of $Q_j$ from $P'_{i_0}$
to $P'_{i_1}$.
Let $M'$ be the union of $P'_i$ and $Q'_j$ for all $i\in I$ and $j\in J$.
It is not hard to see that $M'$ is a mesh. 
We now select a distance function on $M'$ as follows.
Starting with $M'$ we first contract all edges that were contracted
during the production of $H$ from $M$, and then contract edges arbitrarily
until we arrive at a grid $H'$.
We use $H'$ in order to define a distance function $d'$ on $M'$.
It follows that

\hardclaim{2}{$d'(x,y)\ge d(x,y)-\rt{8}t(t-1)$ for all  $x,y\in V(M')$.}

Let $P\in{\cal P}'$, and let $x=x(P)$. We wish to define a path $\phi(P)$
with one end $x$.
If $x\in V(M')$, then $\phi(P)$ is defined to be
the path with vertex-set $\{x\}$;
otherwise we proceed as follows. 
By symmetry between the paths $P_i$ and $Q_j$ we may assume that 
$x\in V(P_i)$.
We claim that $P_i\not\in\cal Q$. 
To prove this claim suppose to the contrary that  $P_i\in\cal Q$. 
Since $x\not\in V(M')$ it follows that
when traversing $P_i$ starting from $Q_0$ we either encounter
$x$ strictly before $Q_{i_0}$, or we encounter $x$ strictly after $Q_{j_0}$.
In either case it follows that $x\in V(B)$, a contradiction.
This proves our claim that $P_i\not\in\cal Q$.
Let $j$ be such that either $x\in V(Q_j)$, or when traversing $P_i$ as above
we encounter $Q_j$, then $x$, and then $Q_{j+1}$.
Then at least one of $Q_j,Q_{j+1}$ belongs to $\cal Q$, for otherwise
$x\in V(B)$, a contradiction (if $x\in V(Q_j)$, then $Q_j\in \cal Q$).
If $Q_j\in \cal Q$, then let $\phi(P)$ be the shortest subpath of $P_i$ from
$x$ to $x'\in V(Q_j)$; 
otherwise let $\phi(P)$ be the shortest subpath of $P_i$ from $x$ to 
$x'\in V(Q_{j+1})$. 
The argument used above to show that $P_i\not\in\cal Q$ now implies
that $x'\in V(M')$.

Let $Y$  be the set of all vertices $y(P)$ over all paths $P\in{\cal P}'$.
Since the graph $B$ is connected, 
by Lemma~\ref{lem:match} there exists a set $\cal R$ of at least
 $\lceil(|{\cal P}'|-1)/4\rceil\ge 32(t(t-1)^6$ disjoint subpaths of $B$,
each with distinct ends in $Y$.
For each $R\in\cal R$ with ends $y_1$ and $y_2$ we define an $M'$-path 
by taking the union  $R\cup P_1\cup\phi(P_1)\cup P_2\cup\phi(P_2)$,
where $P_i\in{\cal P}'$ satisfies $y(P_i)=y_i$.
These paths are pairwise vertex-disjoint.
Since for distinct paths $P,P'\in{\cal P}'$ the vertices $x(P),x(P')$
are at distance at least $\rt{10}t(t-1)$ in $M$, they are at distance at least
$2t(t-1)$ in $M'$ by (2).
By Lemma~\ref{lem:grida} the graph $G$ has a $K_t$ minor grasped
by $M'$, and hence it has a $K_t$ minor grasped by $M$
by Lemma~\ref{lem:controltrans}, as desired.~\qed

\begin{Lemma}
\label{lem:rtmesh}
Let $t\ge1$ be an integer, let $k:=12288(t(t-1))^{12}$,
let $G$ be a graph, and let $M$ be a mesh in $G$ with distance function $d$.
Then either $G$ has a $K_t$ minor grasped by $M$, or there
exist sets $A\subseteq V(G)$ and $Z\subseteq V(M)$ such that
$|A|\le k-1$, $|Z|\le 3k-3$,
and if $x,y$ are the ends of an
$M$-path in $G-A$, then either $d(x,y)<\rt{10}t(t-1)$,
or each of $x,y$
lies at distance at most $\rt{10}t(t-1)-1$ from some vertex of $Z$.
\end{Lemma}

\myproof
We define a relation $R$ on $V(M)$ by saying that $(x,y)\in R$ if
$d(x,y)<\rt{10}t(t-1)$ and apply  Lemma~\ref{lem:spreadpaths} to  the relation
$R$, mesh $M$ and integer $k$.
If the second outcome holds, then the second outcome of the
current lemma holds, and so we may assume that the first outcome of
 Lemma~\ref{lem:spreadpaths} holds.
Thus there exists a set ${\cal P}$ of $k$ pairwise disjoint $M$-paths
that are $R$-semi-dispersed. 
The set $\cal P$ satisfies the hypothesis of Lemma~\ref{lem:semidisperse},
and hence that lemma implies that the graph $G$ has a $K_t$ minor grasped by $M$, as desired.~\qed

\section{Proof of the Flat Wall Theorem}
\label{sec:proof}

We need a somewhat technical lemma before we can begin the proof of Theorem~\ref{thm:main}.  In the proof of 
Theorem \ref{thm:main}, we will find a large flat wall $W'$ in a subgraph, say $G'$, of the original graph $G$.  To find a 
flat wall in $G$ itself, we then consider a smaller subwall $\zz{W''}$ of $W'$.  It is intuitive that $\zz{W''}$ should be flat as well; the near-planarity of $\zz{W'}$ should ensure this.  However, to rigorously consider the sequence of reductions certifying that 
a subwall of $W'$ is flat requires some care.  The next lemma allows us to do so.
It is the main lemma we need for the proof of the Flat Wall Theorem.

\begin{Lemma}
\label{lem:flatwalllike}
Let $G$ be a graph and let $C$ be a cycle in $G$ such that some $C$-reduction of $G$ can be 
drawn in the plane with $C$ bounding a face. Let $W$ be
a subgraph of $G$ and let $D$ be a cycle in $W$ such that $W-V(D)$ is connected and
there exist four internally disjoint paths from 
$V(W)\setminus V(D)$ to $V(C)$ with distinct ends in $V(C)$
such that each intersects $D$ in a (non-null) path.
Then  there exists a separation $(A,B)$ in $G$ such that
\begin{enumerate}
\item[{\rm(1)}] $A\cap B\subseteq V(D)$,
\item[{\rm(2)}] $V(W)\subseteq B$,
\item[{\rm(3)}] $V(C)\subseteq A$, and
\item[{\rm(4)}] some $A\cap B$-reduction of $G[B]$ can be drawn in a disk with $A\cap B$ drawn on the
boundary of the disk in the order determined by $D$.
\end{enumerate}
\end{Lemma}

\def\bd{{\rm bd}}

\begin{proof}
\zz{%
By Theorem~\ref{thm:crossredrend} there exists a $C$-rendition $(\Gamma,\sigma,\pi)$ of $G$,
as defined prior to Theorem~\ref{thm:crossredrend}. In particular, $\Gamma$ is a painting in the unit disk
$\Delta$.
We now define a set $X\subseteq\Delta$ homeomorphic to the unit circle.
The existence of the  four internally disjoint paths from 
$V(W)\setminus V(D)$ to $V(C)$ implies that $D$ is not a subgraph of $\sigma(c)$,
for all $c\in C(\Gamma)$.
Therefore $D$ can be written as $P_1\cup P_2\cup\cdots\cup P_n$, where $n\ge2$ and each $P_i$
is a path with both ends and no internal vertex in $\pi(N(\Gamma))$.
For each $i=1,2,\ldots,n$ the path $P_i$ is a subgraph of $\sigma(c_i)$ for a unique $c_i\in C(\Gamma)$.
Let $\pi(x)$ and $\pi(y)$ be the ends of $P_i$.
Let $X_i$ be the closure of a component of bd$(c_i)\setminus\{x,y\}$  that is disjoint from $N(\Gamma)$.
Thus if $|\widetilde c_i|=3$, then this component is unique, whereas if $|\widetilde c_i|=2$,
then there are two such components. 
\cc{If $|\widetilde c_i|=3$, then let bd$(c_i)\setminus\{x,y\}=\{z_i\}$; otherwise $z_i$ is undefined.} 
Finally let $X=X_1\cup X_2\cup\cdots\cup X_n$.
We will refer to $X$ as the {\em track} of $D$.}

\zz{%
Let $\Delta'$ be the closed disk bounded by $X$, let $A$ be the union of $V(\sigma(c))$
over all $c\in C(\Gamma)$ such that $c\not\subseteq\Delta'$, and  let $B'$ be the union of $V(\sigma(c))$
over all $c\in C(\Gamma)$ such that $c\subseteq\Delta'$.
Then $(A,B')$ is a separation of $G$ that satisfies (1) and (3).
Let $B$ be the union of $B'$ and $V(P_i)$ for all $i=1,2,\ldots,n$  such  that 
$\sigma(c_i)\not\subseteq\Delta'$.
Then $(A,B)$ is also a separation of $G$, and it also satisfies (1) and (3).}

\zz{%
We claim that $(A,B)$ satisfies the conclusion of the theorem.
To prove that we must show that  $(A,B)$ satisfies (2) and (4), and we begin with (4).
\cc{For $i=1,2,\ldots,n$  such  that $\sigma(c_i)\not\subseteq\Delta'$ let $d_i$ be a closed 
disk with $d_i=\overline{c_i}$ if $|\widetilde c_i|=2$ and $c_i\cap X\subseteq d_i\subseteq \overline{c_i}\setminus \{z_i\}$
otherwise.}
Let $\Delta''$ be the union of $\Delta'$ and 
all the \cc{disks $d_i$}.
Then  $\Delta''$ is a closed disk.
Let $\Gamma'$  be the painting defined by  $N(\Gamma')= N(\Gamma)\cap\Delta'$ and
 $U(\Gamma')= U(\Gamma)\cap\Delta''$.
Thus every cell $c\in C(\Gamma')$ is either a subset of $\Delta'$, in which case $c\in C(\Gamma)$,
 or there exists $i=1,2,\ldots,n$ such that \cc{$\sigma(c_i)\not\subseteq\Delta'$ and $c\subseteq d_i$}.
In the former  case we define $\sigma'(c)=\sigma(c)$, and in the latter case we define
 $\sigma'(c)=P_i$.
We define $\pi'$ to be the restriction of $\pi$ to $N(\Gamma')$.
Then $(\Gamma',\sigma',\pi')$ is an $\Omega$-rendition of $G[B]$,
where $\Omega$ is a cyclic ordering of $A\cap B$ and the cyclic order is determined by the order on $D$.
It follows from Theorem~\ref{thm:crossredrend} that $(A,B)$ satisfies (4).}

\zz{%
To prove that $(A,B)$ satisfies (2) we first note that $V(D)\subseteq B$, and so it remains to show
that $V(W)\setminus V(D)\subseteq B$. To that end suppose for a contradiction that 
$V(W)\setminus V(D)\not\subseteq B$.
Since  $W-V(D)$ is connected, it follows that  $W-V(D)$ is a subgraph of
$$\bigcup(\sigma(c)\,:\,c\in C(\Gamma)\hbox{ and }c\not\subseteq\Delta').$$
Let $P_1,P_2,P_3$ be three of the four paths guaranteed by the hypothesis of the lemma.
We may assume that they have a common end in $W-V(D)$.
By considering the ``tracks" of $P_1,P_2,P_3$ (defined similarly as above), we obtain
a planar drawing of the graph $H':=C\cup D\cup P_1\cup P_2\cup P_3$ in which both $C$ and $D$
bound faces. Let $H$ be obtained from $H'$ by 
adding a new vertex in the face bounded by $D$ and
joining it by an edge to every vertex of $D$,
and adding a new vertex in the face bounded by $C$ and joining it by an edge to
every vertex of $C$.
Then $H$ is planar, and yet it has a $K_{3,3}$ subdivision (because each $P_i$ intersects $D$), a contradiction.
This proves that  $(A,B)$ satisfies (2), and hence it satisfies the conclusion of the lemma.}
\end{proof}

Let $H$ be a subgraph of a graph $G$.  
An \emph{$H$-bridge}\?{$H$-bridge} in $G$ is a
connected subgraph $B$ of $G$ such that $E(B)\cap E(H)=\emptyset$
and either $E(B)$ consists of a unique edge with both ends in $H$, or for
some component $C$ of $G\backslash V(H)$ the set $E(B)$ consists of all
edges of $G$ with at least one end in $V(C)$.  The vertices in
$V(B)\cap V(H)$ are called the \emph{attachments}\?{attachments} of $B$.

We are now ready to prove the Flat Wall Theorem, which we restate.

\begin{Theorem}
\label{thm:wst}
Let $r,t \ge 1$ be integers, let $r$ be even, let 
$R=49152t^{24}(\rt{40}t^2+r)$, let $G$ be a graph, and 
let $W$ be an $R$-wall in $G$.
Then either $G$ has a model of a $K_t$ minor grasped by $W$, or there exist
a set $A\subseteq V(G)$ of size at most $12288t^{24}$ and an $r$-subwall
$W'$ of $W$ such that $V(W')\cap A=\emptyset$ and $W'$ is a flat wall
in $G-A$.
\end{Theorem}

\begin{proof}
Let $t, r \ge 1$, and $W$ be given, where  $W$ is an $R$-wall in $G$,
and $R\ge 4\cdot 12288t^{24}(\rt{40}t(t-1))+r)$.
Let $d$ be a distance function on $W$.
By Lemma~\ref{lem:rtmesh} applied to the mesh $W$ and distance function $d$
we may assume that there
exist sets $A\subseteq V(G)$ and $Z\subseteq V(M)$ such that

\hardclaim{1}
{$|A|\le 12288(t(t-1))^{12}$, $|Z|\le 3\cdot  12288(t(t-1))^{12}$,
and if $x,y$ are the ends of a
$W$-path in $G-A$, then either $d(x,y)<\rt{10}t(t-1)$,
or each of $x,y$
lies at distance at most $\rt{10}t(t-1)-1$ from some vertex of $Z$.}


Let the horizontal paths of $W$ be $P_0, \dots, P_R$ and the vertical paths 
$Q_0, \dots, Q_R$.  
A \emph{strip} of $W$ is a subgraph of $W$ consisting of $\rt{40}t(t-1) + r$ 
consecutive horizontal paths of $W$, say $P_{i+1}, \dots, P_{i+\rt{40}t(t-1) + r}$, 
along with every subpath $Q$ of a vertical path of $W$ such that
$Q$ has both ends in 
$V(P_{i+1}) \cup \dots \cup V(P_{i+\rt{40}t(t-1) + r})$. 
By our choice of $R$, there exists a strip $S$ consisting of paths
numbered as above such that 
$S$ contains no vertex of $Z \cup A$.
%
%
We conclude that there exist subwalls $W_1, \dots, W_{t(t-1)}$ 
contained in $S$ satisfying the following for all distinct 
integers $i,j=1,2,\ldots,t(t-1)$:
\begin{itemize}
\item[(2)] $W_i$ is a $\rt{20}t(t-1)+r$-wall  such that the 
horizontal paths of $W_i$ are subpaths of the horizontal paths of the
strip $S$ and the vertical
paths of $W_i$ are subpaths of the vertical paths of $S$,
\item[(3)] if $x \in V(W_i)$ and $y \in V(W_j)$, then  $d(x,y)\ge \rt{10}t(t-1)$,
and
\item[(4)] \rt{$W_i$ is disjoint from the first and last $10t(t-1)$ horizontal paths of $S$.}
\end{itemize}
%
See Figure~\ref{fig2}.
\begin{figure}[htb]
 \centering
 \includegraphics[scale = 1]{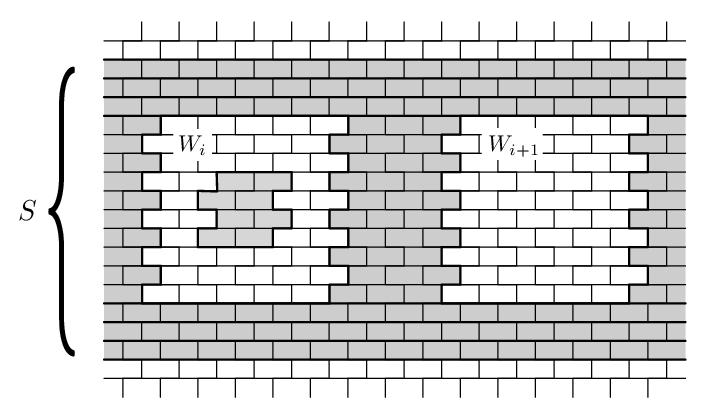}
 \caption{Subwalls of a strip.}
 \label{fig2}
\end{figure}

For $i=1,2,\ldots, t(t-1)$ we define a graph $H_i$.
Let us recall that the corners of a wall were defined at the
end of the first paragraph of Subsection~\ref{sec:weak}.
Let $C_i$ be a cycle with vertex-set the four corners of the wall $W_i$
in the order of their appearance on the outer cycle of $W_i$.
In other words, the cycle $C_i$ may be obtained from the outer cycle
of $W_i$ by suppressing all vertices, except the four corners of $W_i$.
Let $B$ be a $(W-A)$-bridge in the graph $G-A$ with at 
least one attachment in $V(W_i)$, and let $B'$ be obtained from $B$
by deleting all its attachments that do not belong to $V(W_i)$.
The graph $H_i$ is defined as the union of 
the wall $W_i$, the cycle $C_i$ and all graphs $B'$ as above.
We claim that the subgraphs $H_i$ 
are pairwise disjoint.  
To see this, if there exist indices $i$ and $j$ with $i \neq j$ such that 
$H_i$ and $H_j$ share a vertex, then there exists a 
$(W-A)$-bridge with an attachment $x \in V(W_i)$ and $y \in V(W_j)$.  
However then there  exists a $W$-path in $G-A$ with ends $x$ and $y$, 
contrary to (1) and (3), 
\rt{%
because by (4) both $x$ and $y$ are at  distance at least $10t(t-1)$ from
every vertex of $Z$.}


If for all $i=1,2,\ldots,t(t-1)$ the graph $H_i$ has a $C_i$-cross, 
then
the graph $G$ has a $H^1_{t(t-1)}$ minor  such that the underlying grid 
is compatible with the original wall $W$.  
By Lemma~\ref{lem:H1} the graph $H^1_{t(t-1)}$ has a $K_t$ minor 
grasped by the
underlying grid of $H^1_{t(t-1)}$, and hence $G$ has a $K_t$ minor grasped
by $W$ by Lemma~\ref{lem:controltrans}, as desired.

We conclude that we may assume that there exists an index $i$ such that 
the graph $H_i$ does not have a $C_i$-cross. 
By Theorem~\ref{thm:crossreduct} some $C_i$-reduction of $H_i$
can be drawn in the plane with $C_i$ bounding a face.
Let $W'$ be the $r$-wall obtained from $W_i$ by deleting the first and final 
$\rt{10}t(t-1)$ of both the horizontal and vertical paths of $W_i$, and let
$D$ be the outer cycle of $W'$.
%
By Lemma~\ref{lem:flatwalllike} applied to the graph $H_i$,
wall $W'$, cycle $D$ in $W'$ and the cycle $C_i$ there exists
a separation $(X',Y)$ of $H_i$ satisfying (1)--(4) of
Lemma~\ref{lem:flatwalllike}.
Let $X:=X'\cup (V(G)\setminus A\setminus V(H_i))$.
Then $X\cap Y=X'\cap Y\subseteq V(D)$, $V(W')\subseteq V(Y)$,
$V(C_i)\subseteq X$, and some $X\cap Y$-reduction of $G[Y]$ can be
drawn in a disk with $X\cap Y$ drawn on the boundary of the disk.

We claim that $(X,Y)$ is a separation of $G-A$.
To prove this claim suppose for a contradiction that $x\in X\setminus Y$ 
is adjacent to $y\in Y\setminus X$.
Then $y\in V(H_i)$ and $x\not\in V(H_i)$, because $(X',Y)$ is a separation
of $H_i$.
We have $y\not\in V(W_i)\setminus V(W')$, for otherwise $W_i-V(W')$
includes a path from $y$ to $V(C_i)$ disjoint from $V(D)$, 
contrary to the facts that $V(C_i)\subseteq X'$, $y\in Y$ and 
$X'\cap Y\subseteq V(D)$.
It follows that the edge joining $x$ and $y$ belongs to a $(W-A)$-bridge 
of $G-A$, and hence $x$ is an attachment of that $(W-A)$-bridge outside $W_i$.
It follows that this $(W-A)$-bridge includes a $W$-path with one end $x$
and the other end say $x'\in V(W')$.
It follows that $x'$ is at distance at least $\rt{10}t(t-1)$ from $x$
and every vertex in $Z$, contrary to (1).
This proves that $(X,Y)$ is a separation of $G$.

We may choose the pegs of $W'$ in such a way that for every peg $x$
of $W'$ there exists 
a path $P$ in $W$ with one end in $C_i$,
the other end $x$, and otherwise disjoint from $W'$.
It follows that $V(P)\setminus\{x\}\subseteq X \setminus Y$, and hence $x\in X$,
as desired.

Thus the separation $(X,Y)$ is a witness that $W'$ is a flat wall in $G-A$.
We have $V(W')\cap A=\emptyset$, because $W'$ is a subgraph of the
strip $S$, and $S$ was chosen disjoint from $A$.
\end{proof}

\section{A flat wall theorem with few apex vertices}
\label{sec:fewapex}

In this section we prove Theorems~\ref{thm:mainvar} and~\ref{thm:mainvar3}.
The first gives an improved bound on the size of the subset $A$ of vertices.  
It also ensures that the subset $A$ of vertices is highly connected to the resulting wall $W'$, which is useful in applications.
First we need a lemma
and a definition.

\begin{Lemma}
\label{lem:flatsubwall}
Let $G$ be a graph, let $W$ be a flat wall in $G$, and let $W'$ be
a subwall of $W$ disjoint from the outer cycle of $W$.
Then $W'$ is a flat wall in $G$.
\end{Lemma}

\begin{proof}
Let $(A,B)$ be a separation witnessing  that $W$ is a flat wall in $G$.
Let $C$ be the cycle with vertex-set $A\cap B$ such that the
cyclic order of its vertex-set is the one inherited from the
cyclic order of the outer cycle of $W$.
Thus some $C$-reduction of $G[B]\cup C$ can be drawn in the plane
with $C$ bounding a face.
By Lemma~\ref{lem:flatwalllike} applied to the graph $G[B]\cup C$,
wall $W'$, the outer cycle of $W'$ and the cycle $C$ there exists
a separation $(X,Y)$ satisfying (1)--(4) of Lemma~\ref{lem:flatwalllike}.
We may select the pegs of $W'$ in such a way that for every peg $x$ of
$W'$ there exists a path with one end $x$ and the other end in $C$ that
is disjoint from $W'-x$.
Given this choice it follows that every peg of $W'$ belongs to $X$,
and hence the separation $(X,Y)$ shows that the wall $W'$ is flat
in $G$.
\end{proof}

  The next definition makes explicit what we mean by the set $A$ being highly connected to the wall.

\begin{Definition}
Let $W$ be an $r$-wall in a graph $G$ for some positive integer $r \ge 2$.  A \emph{brick} of $W$ is a cycle $C$ which forms the boundary of a finite face 
(that is, a face other than the outer face) in the natural embedding of $W$ in the plane.  Let $A \subseteq V(G)$ and assume $V(W) \cap A = \emptyset$.  A subset $A' \subseteq A$ is \emph{apex-universal for the pair $(W, A)$} if for all $a \in A'$ and for all bricks $C$ of $W$, there exists a path with one end in $V(C)$, one end equal to $a$ which is internally disjoint from $V(W) \cup A$.  
If $A$ is apex-universal for $(W,A)$, then we just say that $A$ is {\em apex-universal} for $W$.
\end{Definition}

We now give the strengthening of Theorem \ref{thm:main}.  

\begin{Theorem}
\label{thm:mainvar2}
Let $t \ge 5$ and $r \ge 3\lceil\sqrt t\rceil$ be integers.
Let $n=12288t^{24}$, $R=r^{2^n}$ and
 $R_0=\yy{49152t^{25}(\rt{40}t+R)}$.
Let $G$ be a graph, and
let $W_0$ be an $R_0$-wall in $G$.
Then either $G$ has a model of a $K_t$ minor grasped by $W_0$, or there exist
a set $A\subseteq V(G)$ of size at most $t-5$ and an $r$-subwall
$W$ of $W_0$ such that $V(W)\cap A=\emptyset$, $W$ is a flat wall
in $G-A$ and $A$ is apex-universal for $W$.
\end{Theorem}

\begin{proof}
By Theorem~\ref{thm:main} we may assume that there exists
a set $A_0\subseteq V(G)$ of size at most $n=12288t^{24}$ and an
$\yy{Rt}$-subwall $\yy{W_1}$ of $W_0$ such that $V(\yy{W_1})\cap A_0=\emptyset$ and
$\yy{W_1}$ is a flat wall in $G-A_0$.
\yy{%
Let $W$ be a subwall of $W_1$ obtained by selecting every $t^{\hbox{th}}$  horizontal and every
$t^{\hbox{th}}$ vertical path of $W_1$.
}%

We fix a subwall $W'$ of $\yy{W}$ and subsets $A' \subseteq \bar{A} \subseteq A_0$ such that 
\begin{itemize}
\item[(1)] $W'$ is a $r^{2^{|\bar{A}| - |A'|}}$-subwall of $\yy{W}$, 
\item[(2)] $W'$ is flat in $G-\bar{A}$, and
\item[(3)] the subset $A'$ is apex-universal for $(W', \bar{A})$.
\end{itemize}
Moreover, we pick $W'$, $A'$, and $\bar{A}$ satisfying (1)-(3) to minimize $|\bar{A}| - |A'|$.  Note that such 
a choice exists by setting $W' = W$, $A' = \emptyset$, and $\bar{A} = A$.  

We claim that $\bar{A}=A'$. To prove that
assume for a contradiction that $\bar{A} \neq A'$.  We define a subwall $W^*$ of $W'$ as follows.  Let $k = r^{2^{|\bar{A}| - |A|-1}}$; thus $W'$ is
a $k^2$-wall.  
Let the vertical and horizontal paths of $W'$ be $V_1, \dots, V_{k^2}$ and $H_1, \dots, H_{k^2}$, respectively.  Let $W^*$ be the $k$-subwall of $W'$ whose horizontal and vertical paths are subpaths of $\{H_{2+i(k-1)}: 1 \le i \le k\}$ and $\{V_{2+i(k-1)}: 1 \le i \le k\}$.  Note that $W^*$ does not intersect the outer cycle of $W$, which will allow us to apply Lemma \ref{lem:flatsubwall} later.  
Exactly one component of $W' -V(W^*)$  contains the outer cycle of $W'$, and every brick of $W^*$ is the outer cycle of a $k$-subwall of $W'$.  
Let $W_1, \dots, W_{(k-1)^2}$ be these $k$-subwalls of $W'$.


Fix a vertex $a \in \bar{A} \setminus A'$.  Assume, as a case, that for all $i\in\{1,2,\ldots,(k-1)^2\}$, there exists a path $P_i$ with one end equal to $a$, one end in $V(W_i)$ and internally disjoint from $V(W') \cup \bar{A}$.  Then we claim that $A' \cup \{a\}$ is apex-universal for $(W^*, \bar{A})$.  
Fix a brick $C$ of $W^*$; let  $i\in\{1,2,\ldots,(k-1)^2\}$ be such that $C$ is the outer cycle of $W_i$.
Thus, by extending $P_i$ through $W_i$, we can find a path from $V(C)$ to $a$ with no internal vertex in $V(W^*) \cup A$.  
Similarly, if  $a' \in A'$, then there exists a path $P'$ from $a'$ to some (in fact, every) brick of $W_i$ with no internal vertex in $V(W) \cup A$.  Thus, again we can extend $P'$ through $W_i$ to find a path from $V(C)$ to $a'$ which has no internal vertex in $V(W^*) \cup A$.  We conclude that $A' \cup \{a\}$ is apex-universal for $(W^*, \bar{A})$.  
It follows now by Lemma \ref{lem:flatsubwall} that $W^*$, $A' \cup \{a\}$, and $\bar{A}$ satisfy (1)--(3), contrary to our choice to minimize $|\bar{A}| - |A'|$.

Thus there exists an index $i\in\{1,2,\ldots,(k-1)^2\}$ such that there does not exist a path with one end equal to $a$ and one end in $V(W_i)$ which is internally disjoint from $V(W') \cup \bar{A}$.  As every brick of $W_i$ is a brick of $W'$, we see that $A'$ is apex-universal for $(W_i, \bar{A} \setminus \{a\})$.  We claim as well that $W_i$ is flat in $G- (\bar{A} \setminus \{a\})$.  By Lemma \ref{lem:flatsubwall}, $W_i$ is flat in $G-\bar{A}$.  Let $(X,Y)$ be a separation of $G-\bar{A}$ as in the definition of flat wall, chosen with $|Y|$ minimum.
The minimality of $Y$ implies that for every $y\in Y \setminus X$ there exists
a path in $G[Y]-X$ with one end $y$ and the other end in $V(W_i)$.  Note that $W' - V(W_i)$ is connected; it follows that $V(W') \setminus V(W_i)$ is contained in $X$.  We conclude that $a$ has no neighbor in $Y \setminus X$, lest there exist a path from $a$ to $W_i$ avoiding the vertices of $V(W') \setminus V(W_i)$. 
Consequently, $(X \cup \{a\}, Y)$ is a separation of $G- (\bar{A} \setminus \{a\})$ that proves that the wall $W_i$ is flat in $G-(\bar{A} \setminus \{a\})$.  
It follows that $W_i$, $A'$, and $\bar{A} \setminus \{a\}$ satisfy (1)--(3), again contrary to our choice.
This proves our claim that $\bar{A}=A'$.

We conclude that $W'$ is an $r$-subwall of $W$ which is flat in $G-\bar{A}$.  To complete the proof, it suffices to show that $|\bar{A}| \le t-5$.  Assume not, and that $|\bar{A}| \ge t-4$; let $a_1, \dots, a_{t-4}$ be $t-4$ distinct vertices in $\bar{A}$.  
By the assumption that $r \ge 3\lceil\sqrt t\rceil$ we can choose bricks $C_1,C_2,\ldots,C_t$ in $W'$ such that each of them is disjoint from the outer
cycle of $W'$ and every two distinct bricks in the family are separated by a vertical or horizontal path of $W'$.
For $i\in \{1,2,\ldots,t\}$ and $x\in \bar{A}$ there exists
a path $P^i_x$ from $x$ to $V(C_i)$, internally disjoint from $V(W')\cup \bar{A}$.
For $i\in  \{1,2,\ldots,t\}$ let $X'_i$ be the union of
$V(C_i)$ and all the sets $V(P^i_x)\setminus \bar{A}$ for $x\in \bar{A}$.
For $i\in  \{1,2,\ldots,t-4\}$ let $X_i=X_i'\cup\{a_i\}$, and for
$i\in\{t-3,t-2,t-1,t\}$ let $X_i=X'_i$.
The sets $X_i$ induce connected graphs, and we claim that they are pairwise
disjoint. Indeed, to see that it suffices to argue that
for distinct $i,j\in  \{1,2,\ldots,t\}$ and not necessarily distinct $x,y\in \bar{A}$
the paths $P^i_x-x$ and $P^j_y-y$ are disjoint.
But if those two paths intersect, then there exists a path $P$ in $G-\bar{A}$
from $C_i$ to $C_j$ that is internally disjoint from $W'$.
However, the existence of $P$ contradicts the flatness of $W'$.
To see this, let $(X,Y)$ be a separation of $G-\bar{A}$ as in the definition
of flat wall, and let $s_1,s_2,t_1,t_2\in X\cap Y$ be distinct vertices
appearing on the outer cycle of $W'$ in the order listed.
It follows that $W'\cup P$ has two disjoint paths, one with ends
$s_1$ and $t_1$, and the other with ends $s_2$ and $t_2$.
However, that contradicts the fact that some $X\cap Y$-reduction
of G[Y] can
be drawn in a disk with the vertices $s_1,s_2,t_1,t_2$ drawn on the
boundary of the disk in order.
This proves that the sets $X_i$ are pairwise disjoint.

The sets $X_i$ can be modified,
\yy{%
using the vertical and horizontal paths of $W_1$ that are not part of $W$,
}%
 to give model of a $K_t$ minor
grasped by $W'$, and hence grasped by $W_0$.
The only thing that is missing are edges between the sets
$X_{t-3},X_{t-2},X_{t-1},X_{t}$, and those can be supplied
by enlarging these sets
using horizontal and vertical paths of $W'$ that are disjoint from
all the cycles $C_i$. We omit the details, which are easy.
\end{proof}

Let $G$ be a graph, let $C$ be a cycle in $G$, and let $J$ be a $C$-reduction
of $G$ obtained by successively performing elementary $C$-reductions determined
by separations $(A_1,B_1), (A_2,B_2),\allowbreak\ldots,(A_k,B_k)$ in the order listed.
More precisely, let $G_0:=G$, for $i=1,2,\ldots,k$ let $G_i$
be obtained from $G_{i-1}$ by the elementary $C$-reduction determined by
$(A_i,B_i)$, and let $J=G_k$.
If $J$ can be drawn in the plane with $C$ bounding a face, then
we say that $(B_1,B_2,\ldots,B_k)$ is a
\emph{$C$-reduction sequence}\?{$C$-reduction sequence} for $G$.
Given a $C$-reduction sequence $(B_1,B_2,\ldots,B_k)$, the original
separations may be recovered by letting $A_i$ be the set of all
vertices $v\in V(G_{i-1})$  such that either $v\not\in B_i$, 
or $v\in B_i$ and $v$ has a neighbor in $V(G_{i-1})\setminus B_i$.

We now prove Theorem~\ref{thm:mainvar3}, which we restate.

\begin{Theorem}
\label{thm:mainvar4}
Let $r\ge2$ and $t \ge 5$ and 
be integers,
let $n=12288t^{24}$ and
 $R_0=49152t^{24}(\rt{40}t^2+(rt)^{2^n})$ and let $G$ be a graph
with no $K_t$ minor.
If $G$ has an $R_0$-wall,
then there exist
a set $A\subseteq V(G)$ of size at most $t-5$ and an $r$-wall
$W$ in $G$ such that $V(W)\cap A=\emptyset$ and $W$ is a flat wall
in $G-A$.
Furthermore, if $(X,Y)$ is a separation as in the definition of flat wall,
then the graph $G[Y]$ has no $(R_0+1)$-wall.
\end{Theorem}

\begin{proof}
There exists a separation $(X_0,Y_0)$ of $G$ of order at most $t-2$
such that the graph $G[Y_0]$ has an $R_0$-wall, because the separation
$(\emptyset,V(G))$ has said property.
We may choose such a separation such that $Y_0$ is minimal with
respect to inclusion.
Let $G_0$ denote the graph $G[Y_0]$.
By Theorem~\ref{thm:mainvar2} applied to the graph $G_0$, an
$R_0$-wall in $G_0$ and the integer $rt$ in place of $r$
 we may assume that  there exist
a set $A\subseteq V(G_0)$ of size at most $t-5$ and an $rt$-wall
$W$ in $G_0$ such that $V(W)\cap A=\emptyset$ and $W$ is a flat wall
in $G_0-A$.
Let $(X_0',Y_0')$ be a separation as in the definition of flat wall.

\rt{%
Let us select $W, X_0', Y_0'$  as stated in the previous paragraph, 
and subject to that in such a way that} 
$Y_0'$ is minimal with respect to inclusion.
Let $W_1,W_2,\ldots,W_t$ be disjoint $r$-subwalls of $W$ such that
each  is disjoint from the outer cycle of $W$ and every two of them are
separated by a vertical or horizontal path of $W$.
Let $i=1,2,\ldots,t$. By Lemma~\ref{lem:flatsubwall} the wall $W_i$
is flat in $G_0-A$; let $(A_i,B_i)$ be the corresponding separation.
We claim that the sets $B_1,B_2,\ldots,B_t$ are pairwise disjoint.
Indeed, otherwise there exists a $W$-path in $G_0-A$ with ends in
different subwalls $W_i$, a contradiction similarly as in the proof
of Theorem~\ref{thm:wst} or Theorem~\ref{thm:mainvar2}.
Since $|X_0\cap Y_0|\le t-2$ we may assume that $B_1$ is
disjoint from $X_0\cap Y_0$.
It follows that $(A_1\cup X_0,B_1)$ is a separation of $G-A$,
and hence the wall $W_1$ is flat in $G-A$.

It remains to show that $G[B_1]$ has no $(R_0+1)$-wall.
To that end suppose for a contradiction that $W_0$ is an $(R_0+1)$-wall
in $G[B_1]$, let $Y_1,Y_2,\ldots,Y_k$ be an $A_1\cap B_1$-reduction
sequence for $G[B_1]$, and let $(X_i,Y_i)$ be the corresponding separations.
Thus $Y_i\subseteq B_1$ for every $i=1,2,\ldots,k$.
We may assume that the $A_1\cap B_1$-reduction sequence is chosen with
$k$ maximum.
For each $i=1,2,\ldots,k$ either 
every vertex of $W_0$ of degree three except possibly one
belongs to $X_i$, or  
every vertex of $W_0$ of degree three except possibly one
belongs to $Y_i$.
Let us assume first that the latter holds for some index $i\in\{1,2,\ldots,k\}$.
Then $G[Y_i]$ has an $R_0$-wall (a subwall of $W_0$) and
$(V(G)\setminus(Y_i\setminus X_i),Y_i\cup A)$ is a separation of $G$
of order at most $t-2$ that contradicts the choice of $(X_0,Y_0)$, because
$A\subseteq Y_0$, $Y_i\subseteq B_1\subseteq Y_0$ and
at least one corner of $\zz{W_1}$ belongs to $Y_0\setminus Y_i$.
It follows that for each $i=1,2,\ldots,k$  
every vertex of $W_0$ of degree three except possibly one
belongs to $X_i$.

Let $J$ be the $A_1\cap B_1$-reduction of $G[B_1]$ that arises by applying
the $A_1\cap B_1$-reduction sequence $Y_1,Y_2,\ldots,Y_k$. We may assume
that $J$ is drawn in a disk $\Delta$ in such a way that the vertices of $A_1\cap B_1$
are drawn on the boundary of $\Delta$ in the order determined by the outer 
cycle of $W_1$.
Since the graph $G[B_1]$ has the $R_0$-wall $W_0$, it has an $R_0\times R_0$-grid
minor such that no branch set of the minor 
\rt{that corresponds to a vertex of degree four of the grid minor}
is a subset of $Y_i\setminus X_i$. 
\rt{With the possible exception of vertices of the outer cycle}
such a grid minor is preserved under the
$A_1\cap B_1$-reductions, which implies that $J$ has an $(R_0-2)\times (R_0-2)$-grid
minor, and hence an $(R_0/2-1)$-wall.
Thus $J$ has an $rt$-wall $W'$ such that the face bounded by the outer cycle
of $W'$ includes the boundary of $\Delta$.
Let $D'$ be the outer cycle of $W'$.
It follows from the maximality of $k$
that there exist four internally disjoint paths in $J$ from $W'-V(D')$ to $A_1\cap B_1$
with distinct ends in $A_1\cap B_1$.
By changing \zz{the paths} if necessary we may assume that each of these four paths
intersects $D'$ in a path.
Let $W''$ be an $rt$-wall in $G[B_1]$ obtained by
\rt{%
converting the wall $W'$ into one in $G[B_1]$.
This is done mostly by replacing edges of $W'$
that do not belong to $G$ by corresponding paths in $G[Y_i]$ for some
$i\in\{1,2,\ldots,k\}$. 
Likewise, the four internally disjoint paths in $J$ can be converted to paths in $G[B_1]$.}
By Lemma~\ref{lem:flatwalllike} the wall $W''$ is flat in \zz{$G[B_1]$, and hence in} $G_0-A$;
thus the corresponding separation
contradicts the choice of $(X_0',Y_0')$.
Thus $G[B_1]$ has no $(R_0+1)$-wall, as desired.
\end{proof}

\section{An Algorithm}
\label{sec:algo}

We need algorithmic versions of Lemmas~\ref{lem:spreadpaths} 
and~\ref{lem:diversepaths}.
In order for those algorithms to run efficiently we need to make
some assumptions about the computability of the relation $R$.
It seems best to do so in the context of our application, namely
when $M$ is a mesh in the graph $G$ and
$(x,y)\in R$ if and only if $d(x,y)<l$ for some integer $l$, where
$d$ is a distance function on $M$.
Let us recall that the notion of a distance function was defined
at the beginning of Section~\ref{sec:mesh} by saying that
$d(x,y)$ is the distance of $f(x)$ and $f(y)$ in $H$, where $H$
is a grid minor of $M$ and $f:V(M)\to V(H)$ describes the contraction.
We will refer to $f:V(M)\to V(H)$ as a
\emph{grid contraction function}\?{grid contraction function}.
It is clear that given a grid contraction function $f$, the value
$d(x,y)$ can be computed in constant time for any $x,y\in V(M)$.
Thus we will use a grid contraction function to represent the distance
function on $M$.
We  assume that for each $x\in V(M)$ we store the value $f(x)$,
and that for each $u\in V(H)$ we store $f^{-1}(u)$ as a list.

Let an integer $l\ge 0$ be fixed, and let $(x,y)\in R$ if and only 
if $d(x,y)<l$.
We need to clarify one issue about the sets $R(x)$.
Let us recall that $R(x)$ denotes
the set of all $y\in X$ such that $(x,y)\in R$.
If $x\in V(M)$, then
$R(x)$ can be written as $\bigcup_{v\in V_1\cup V_2}f^{-1}(v)$ for some
sets $V_1,V_2\subseteq V(H)$, where $|V_1|\le(2l-1)^2$
and $V_2$ is the union of the vertex-sets of at most $2l-1$ vertical 
and at most $2l-1$ horizontal paths of $H$.
To see this let $V_1$ be the set of all vertices $v\in V(H)$ such that
there is a curve in the plane connecting $v$ and $f(x)$ that intersects
$H$ at most $l$ times and does not use the outer face of $H$,
and $V_2$ is defined analogously using curves that use the outer face of $H$.

The following is an algorithmic version of Lemma~\ref{lem:spreadpaths}.
The conclusion is slightly weaker in order to save on running time.

\begin{Lemma}
\label{alg:spreadpaths}
There exists an algorithm with the following specifications.\\
{\bf Input:} A graph $G$ on $n$ vertices and $m$ edges, 
 integers $k,l\ge1$, and a mesh $M$ in $G$ with grid contraction function
$f:V(M)\to V(H)$ giving rise to a distance function $d$ on $M$.
For $x,y\in V(M)$ let $(x,y)\in R$ if and only if $d(x,y)<l$.\\
{\bf Output:} Either $k$ disjoint $R$-semi-dispersed $M$-paths,  or
sets $A \subseteq V(G)$ and $Z \subseteq V(M)$ with $|A| \le k-1$
and $|Z| \le 3k-3$ such that every $M$-path $P$ in
$G-A$ with ends $x$ and $y$ either satisfies $d(x, y)\le 2l-2$ or both
$x, y \in \bigcup_{z \in Z} R(z)$.\\
{\bf Running time:} $O(\xx{\min\{n,k\}}m+n)$.
\end{Lemma}

\begin{proof}
We may assume that $G$ has no isolated vertices (by deleting them).
If $l$ is at least the number of vertical or horizontal paths in $M$,
then $A:=\emptyset$ and any one-element set $Z\subseteq V(M)$
(or $Z=\emptyset$ if $k=1$ and no $M$-path with ends far apart exists)
satisfy the second condition of the output requirement.
Thus we may assume that $l^2=O(m)$.

The algorithm will proceed in at most $3k$ iterations.
At the beginning of each iteration there will be $M$-paths 
$P_1,P_2,\ldots,P_s$ and $Q_1,Q_2,\ldots,Q_p$ as in the proof 
of Lemma~\ref{lem:spreadpaths} with ends denoted in the same way.
Let $A,Z,W$ be defined as in the proof of Lemma~\ref{lem:spreadpaths}.
At the start of the first iteration we have $s=p=0$;
thus $A=Z=W=\emptyset$.
Throughout the algorithm the set $W$ will be of the form 
$\bigcup_{v\in V}f^{-1}(v)$ for some
$V\subseteq V(H)$, and will be presented by marking the elements of $V$.


For the purpose of this paragraph and the next let us say that
a \emph{good path} is 
an $M$-path $S$ in $G-A$ with ends $x,y$,
where $x\in V(M) \setminus W$ and $(x,y)\not\in R$.
We say that $S$ is \emph{very good} if it is good and $d(x,y)\ge 2l-1$.
At the beginning of each iteration we either find a good path,
or establish that no very good path exists.
We do so by running the following subroutine for every $M$-bridge
$B$ of the graph $G-A$.
In the subroutine
we first test whether $B$ has an attachment $x\in V(M) \setminus W$.
If not, then $B$ does not include a good
path and we return that information.
Otherwise we test whether $B$ has an attachment $y$ at distance at least
$l$ from $x$;
if we find one, then a path in $B$ from $x$ to $y$ is a good path,
and we return it.
On the other hand, if all attachments of $B$ belong to $R(x)$, then
$B$ includes no very good path, and we return that information.
This completes the description of the subroutine. 
It is clear that each call takes time $O(|E(B)|)$, and that if no 
call to the subroutine returns a good path, then no very good path exists.
Thus we either find a good path,
or establish that no very good path exists
in time $O(m)$.

If no very good path exists, then the sets $A$ and $Z$ satisfy the 
specifications of the algorithm.
We output those sets and terminate the algorithm.
If we find a good path $S$, then we modify the paths $P_i$ and $Q_i$
as in the proof of Lemma~\ref{lem:spreadpaths}
by either adding a new path $P_{s+1}$ and keeping \xx{all but one of}
the old paths $Q_i$, 
or by adding two new paths $P_{s+1},P_{s+2}$ and
discarding one old path \xx{$P_i$ and one old path} $Q_i$, or by adding a new path $Q_{p+1}$.
In each case the quantity $2s+p$ increases by one.
We update the sets $A$, $Z$ and $W$.
The set $W$ will be updated by marking $f(v)$ for every vertex $v$
that is being added to $W$. 
For every vertex that is being added to $Z$ this involves marking
at most $(2l-1)^2$ vertices of $H$ and the vertex-sets of at most
$2(l-1)$ vertical and at most $2(l-1)$ horizontal paths of $H$.
\xx{Similarly, we unmark vertices that are being deleted from $W$.}
The marking of vertical and horizontal paths will be done implicitly,
so that the total time spent on marking during each iteration will be $O(l^2)$.
If $s\ge k$ we output the paths $P_1,P_2,\ldots,P_k$ 
and terminate the algorithm;
otherwise we go to the next iteration.
The second step of the iteration described in this paragraph takes time
$O(l^2+n)=O(m)$.

Since the quantity $\xx{|Z|=2s+p}$ increases during each iteration and $p\le s$,
the algorithm will terminate after at most $\xx{\min\{n,3k\}}$ iterations.
Thus the running time is as claimed.
\end{proof}

Likewise there is a  version of Lemma~\ref{lem:diversepaths} with
a similar proof, which we omit.

\begin{Lemma}
\label{alg:diversepaths}
There exists an algorithm with the following specifications.\\
{\bf Input:} A graph $G$ on $n$ vertices and $m$ edges,
 integers $k,l\ge0$, and a mesh $M$ in $G$ with grid contraction function
$f:V(M)\to V(H)$ giving rise to a distance function $d$ on $M$.
For $x,y\in V(M)$ let $(x,y)\in R$ if and only if $d(x,y)<l$.\\
{\bf Output:} Either $k$ disjoint $R$-dispersed $M$-paths,  or
sets $A \subseteq V(G)$ and $Z \subseteq V(M)$ with $|A| \le k-1$
and $|Z| \le 3k-3$ such that for every $M$-path $P$ in
$G-A$ its  ends can be denoted by $x$ and $y$ such that either
$d(x, y)\le 2l-2$ or  
$x \in \bigcup_{z \in Z} R(z)$.\\
{\bf Running time:} $O(\xx{\min\{n,k\}}m+n)$.
\end{Lemma}

\begin{Lemma}
\label{alg:grida}
There is an algorithm with the following specifications.\\
{\bf Input:} A graph $G$ on $n$ vertices and $m$ edges, 
an integer $t\ge2$, a mesh in $G$ with grid
contraction function $f:V(M)\to V(H)$ giving rise to a distance function
$d$ on $M$,
a set $X \subseteq V(M)$ with
$|X| = 64(t(t-1))^6$ such that $d(x, y) \ge 2t(t-1)$ for all  $x, y \in X$,
and a matching $F\subseteq E(G) \setminus E(M)$ in $G$ of size $32(t(t-1))^6$ with
vertex-set $X$.\\
{\bf Output:} A model of $K_t$ grasped by $M$.\\
{\bf Running time:} $O(m+n)$.
\end{Lemma}

\begin{proof}
This follows from the proof of Lemma~\ref{lem:grida}, because 
it is easy to convert the standard proofs of Lemmas~\ref{lem:epintervals}
and~\ref{lem:es} into algorithms with running times $O(k^2)=O(m)$,
where $k$ is as in those lemmas.
\end{proof}

\begin{Lemma}
\label{alg:rtmesh}
There exists an algorithm with the following specifications.\\
{\bf Input:} A graph $G$ on $n$ vertices and $m$ edges, 
an integer $t\ge2$, 
and a mesh in $G$ with grid
contraction function $f:V(M)\to V(H)$ giving rise to a distance function
$d$ on $M$.\\
{\bf Output:}
For $k_0:=12288(t(t-1))^{12}$
either a model of $K_t$ in $G$ grasped by $M$, or 
sets $A\subseteq V(G)$ and $Z\subseteq V(M)$ such that
$|A|\le k_0-1$, $|Z|\le 3k_0-2$,
and if $x,y$ are the ends of a
$M$-path in $G-A$, then either $d(x,y)<\rt{20}t(t-1)$,
or each of $x,y$
lies at distance at most $\rt{10}t(t-1)-1$ from some vertex of $Z$.\\
{\bf Running time:} $O(t^{24}m+n)$
\end{Lemma}

\begin{proof}
The algorithm follows the proof of Lemma~\ref{lem:rtmesh}.
We first apply the algorithm of Lemma~\ref{alg:diversepaths} to the
graph $G$, mesh $M$ and integers $l=2t(t-1)$ and $k=32(t(t-1))^6$.
If the algorithm returns $k$ disjoint dispersed $M$-paths,
then we use the algorithm of Lemma~\ref{alg:grida} to output a model
of $K_t$ grasped by $M$ and stop.
We may therefore assume that the algorithm of 
Lemma~\ref{alg:diversepaths} returns sets $A\subseteq V(G)$ and
$Z\subseteq V(M)$ such that $|A|\le k-1$, $|Z|\le 3k-3$, and for
every $M$-path $P$ in $G-A$ its ends may be denoted by $x$ and $y$ such that
either $d(x,y)\le 4t(t-1)-2$ or $d(x,z)\le 2t(t-1)-1$ for some $z\in Z$.
Next we apply the algorithm of Lemma~\ref{alg:spreadpaths} to the
graph $G$, mesh $M$ and integers $l=\rt{10}t(t-1)$ and $k_0$.
If the algorithm returns sets $A$ and $Z$, then we return those sets and stop.
We may therefore assume that the algorithm of Lemma~\ref{alg:spreadpaths}
returns a set of $k_0$ pairwise disjoint semi-dispersed $M$-paths.
We use the argument of the proof of Lemma~\ref{lem:rtmesh} to use the
paths to construct a matching to which we can apply the algorithm
of Lemma~\ref{alg:grida} to output a model
of $K_t$ grasped by $M$.
\end{proof}

The following is an algorithm of \zz{Kawarabayashi, Li and Reed~\cite{KawLiRee}} stated using our
terminology.

\begin{Theorem}
\label{alg:crossreduct}
There is a polynomial-time algorithm with the following specifications.\\
{\bf Input:} A graph $G$ with $n$ vertices and $m$ edges
and a cycle $C$ in $G$.\\
{\bf Output:} Either a $C$-cross in $G$, or 
\zz{a $C$-rendition.}\\
{\bf Running time:} $O(n+m)$.
\end{Theorem}

Let us remark that the algorithm 
\zz{of Kawarabayashi, Li and Reed~\cite{KawLiRee} is  formulated in terms of $C$-reductions, which is equivalent to $C$-renditions
 by Theorem~\ref{thm:crossredrend}.}

Our last lemma is an algorithmic version of Lemma~\ref{lem:flatwalllike}.

\begin{Lemma}
\label{alg:flatwalllike}
There exists an algorithm with the following specifications.\\
{\bf Input:}
A graph $G$ on $n$ vertices and $m$ edges, a subgraph $W$ of $G$, 
a cycle $C$ in $G$,
a cycle $D$ in $W$ such that $W-V(D)$ is connected, 
four internally disjoint paths from
$V(W) \setminus V(D)$ to $V(C)$ with distinct ends in $C$
such that each intersects $D$ in a path,
\zz{and a $C$-rendition of $G$.\\}
{\bf Output:} A separation $(A,B)$ in $G$ satisfying {\rm(1)--(4)}
of Lemma~\ref{lem:flatwalllike}
\zz{and an $\Omega$-rendition of $G[B]$,
where $\Omega$ is a cyclic ordering of $A\cap B$ and  the cyclic order is determined by the order on  $D$.\\}
{\bf Running time:} $O(n+m)$.
\end{Lemma}

\begin{proof}
\zz{Let $(\Gamma,\sigma,\pi)$ be a $C$-rendition of $G$.
We construct a track of $D$ as in the proof of Lemma~\ref{lem:flatwalllike}.
Using the track we construct the separation $(A,B')$, and then modify it 
to the separation $(A,B)$, as in the proof of Lemma~\ref{lem:flatwalllike}.
Finding the original separation takes time $O(n+m)$,
and the modifications take time $\sum_{i=1}^nO(|E(\sigma(c_i))|)$.}
Thus the total running time is $O(n+m)$.
\end{proof}

We are finally ready to describe our main algorithm.

\begin{Theorem}
\label{alg:wst}
There is an algorithm with the following specifications.\\
{\bf Input:} A graph $G$ on $n$ vertices and $m$ edges,
integers $r,t\ge1$,
and an $R$-wall $W$ in $G$, where
$R=49152t^{24}(\rt{60}t^2+r)$.\\
{\bf Output:} 
Either a model of a $K_t$ minor in $G$ grasped by $W$, or 
a set $A\subseteq V(G)$ of size at most $12288t^{24}$ and an $r$-subwall
$W'$ of $W$ such that $V(W')\cap A=\emptyset$ and $W'$ is a flat wall
in $G-A$.
In the second alternative the algorithm also returns a separation $(A,B)$
as in the definition of flat wall, 
\zz{and an $\Omega$-rendition of $G[B]$,
where $\Omega$ is a cyclic ordering of $A\cap B$ and  the cyclic order is determined by the order on 
the outer cycle of $W'$.}\\
{\bf Running time:} $O(t^{24}m+n)$.
\end{Theorem}

\begin{proof}
We compute a grid contraction function $f:V(W)\to V(H)$ and apply
the algorithm of Lemma~\ref{alg:rtmesh} to the graph $G$, mesh $W$, 
function $f$, and integer $t$.
If the algorithm returns a model of $K_t$ grasped by $W$, then we return
that model and stop.
We may therefore assume that the algorithm of Lemma~\ref{alg:rtmesh}
returned sets $A\subseteq V(G)$ and $Z\subseteq V(M)$ such that
$|A|\le 12288(t(t-1))^{12}$, $|Z|\le 3\cdot  12288(t(t-1))^{12}$,
and if $x,y$ are the ends of an
$M$-path in $G-A$, then either $d(x,y)<\rt{20}t(t-1)$,
or each of $x,y$
lies at distance at most $\rt{10}t(t-1)-1$ from some vertex of $Z$.
We define strips similarly as in the proof of Theorem~\ref{thm:wst},
except that strips will now consist of $\rt{60}t(t-1)+r$ consecutive paths.
We construct walls $W_1,W_2,\ldots,W_{t(t-1)}$, but this time each will
be a $(\rt{40}t(t-1)+r)$-wall,  they will be pairwise at distance at least
$\rt{20}t(t-1)$,
\rt{%
and each will be disjoint from the first and last $10t(t-1)$ paths of the strip.}
We construct the graphs $H_i$ and cycles $C_i$ as in the proof of
Theorem~\ref{thm:wst}, and apply the algorithm of 
Theorem~\ref{alg:crossreduct} to each.
If each of them has a $C_i$-cross, then we use those crosses to
construct a model of $K_t$ grasped by $W$, as in the proof of
Theorem~\ref{thm:wst}.
On the other hand if some $H_i$ has 
\zz{a $C_i$-rendition,}
then we apply the algorithm of Lemma~\ref{alg:flatwalllike}
to $H_i$, wall $W_i$, its outer cycle and the \zz{$C_i$-rendition}
to produce a separation $(X',Y)$ satisfying (1)--(4) of
Lemma~\ref{lem:flatwalllike}
\zz{and an $\Omega$-rendition of $G[Y]$,
where $\Omega$ is a cyclic ordering of $X'\cap Y$ and  the cyclic order is determined by the order on 
the outer cycle of $W_i$.}
Finally, we convert $(X',Y)$ to a required separation of $G$ as in
the proof of Theorem~\ref{thm:wst}.
\end{proof}



\bigskip

\begin{samepage}
{\Large \bf Acknowledgment}\\
\\
We would like to thank Ann-Kathrin Elm from the University of Hamburg for pointing out a couple of
inaccuracies in an early draft of this article.
\zz{We would also like to acknowledge that since the first version of this article has been posted,
Chuzhoy~\cite{ChuFlat} dramatically improved the bound in Theorem~\ref{thm:mainvar}
to $R_0=\Omega(t(t+r))$.}
\end{samepage}


\section{Appendix: Characterizing graphs with no $C$-cross}

In this section, we present a proof of Theorem \ref{thm:crossreduct} which characterizes when a given graph $G$ containing a cycle $C$ has a $C$-cross.  The proof is due to Robertson and Seymour \cite{RS9}.  


Let $G$ be a graph and $C$ a cycle in $G$.  
We first prove the easy ``if" implication. We have noted earlier
that if $H$ is an elementary $C$-reduction of $G$, then $H$ contains a $C$-cross if and only if $G$ does as well.  Let $G'$ be any $C$-reduction of $G$.  If $G'$ can be drawn in the plane with $C$ bounding the infinite face, then by planarity, there does not exist a $C$-cross in $G'$.  Consequently, there does not exist a $C$-cross in $G$ as well.  

We now prove the ``only if" implication by induction on  $|V(G)|+|E(G)|$.
If $G=C$, then the theorem clearly holds, and so we may assume that $G\ne C$ and that $G$ has no \yy{$C$-}cross.
We may assume that $G$ is simple, because deleting loops and parallel edges does not change the validity of either of the statements in the theorem. 
If $G$ has an elementary $C$-reduction, then the theorem follows by induction applied to that $C$-reduction.
Thus we may assume that $G$ has no elementary $C$-reduction. Therefore

\myclaim{cl:nosep}{$G$ has no separation $(A,B)$ of order at most three with  $V(C) \subseteq A$ and $B\setminus A\ne\emptyset$,}

\noindent
because if such a separation exists, then choosing one with $|A\cap B|$ minimum gives a separation that determines an elementary $C$-reduction of $G$, 
a contradiction.

Define a \emph{tripod} \?{tripod} as a union of paths $P_1, P_2, P_3, Q_1, Q_2, Q_3$ satisfying the following. The paths $P_1, P_2, P_3$ have a common end $v \in V(G) \setminus V(C)$ and are otherwise pairwise disjoint.  Each $P_i$, $1 \le i \le 3$ has exactly one vertex in $V(C)$, call it $x_i$, and $x_i$ is an end of $P_i$.  The paths $Q_1, Q_2, Q_3$ have a common end $u \in (V(G) \setminus V(C)) $, $u \neq v$, and are otherwise pairwise disjoint.  For every $1 \le i \le 3$, $Q_i$ has an end $y_i \in V(P_i) -\{ v\}$ and $Q_i$ is otherwise disjoint from $P_1\cup P_2\cup P_3$. 

\myclaim{cl:tripod}{The graph $G$ does not contain a tripod.}

To prove~\refclaim{cl:tripod}
assume there exists a tripod $T$ and let the paths $P_1, P_2, P_3, Q_1, Q_2, Q_3$ and the vertices 
$x_1, x_2, x_3, u,\allowbreak v, y_1, y_2, y_3$ be 
labeled as in the definition of a tripod.  
For $i=1,2,3$ let $L_i$ be the subpath of $P_i$ with ends $x_i$ and $y_i$, and let $R_i$ be the  subpath of $P_i$ with ends $v$ and $y_i$.
Let $X=V(R_1\cup R_2\cup R_3\cup Q_1\cup Q_2\cup Q_3)$.
By~\refclaim{cl:nosep} there exist four disjoint paths from $X$ to $V(C)$, and by a standard ``augmenting path" argument 
(cf.~\cite[Section~3]{Diestel}) those paths can be chosen such that three of them have ends in $\{y_1,y_2,y_3\}$ and (possibly different)
three of those paths have ends in $\{x_1,x_2,x_3\}$.
Thus by possibly replacing the paths $L_1,L_2,L_3$ by a different set of disjoint paths we may assume that  there exists a path $Q$
with one end in $X\setminus\{y_1,y_2,y_3\}$ and the other end in $V(C)-\{x_1,x_2,x_3\}$ that is disjoint from $T$ except for one of its ends.
It follows that $T\cup Q$ includes a $C$-cross, a contradiction, which proves~\refclaim{cl:tripod}.
\bigskip

Let us recall that $H$-bridges were defined prior to Theorem~\ref{thm:wst} and $H$-paths were defined at the beginning of Section~\ref{sec:mpaths}.
If $P$ is a $C$-path, then a $C\cup P$-bridge is {\em unstable} if all its attachments belong to $V(P)$, and {\em stable} otherwise.

\myclaim{cl:stable}{There exists a $C$-path $P$ in $G$ such that every $C\cup P$-bridge is stable.}

To prove~\refclaim{cl:stable} we first note that since $G\ne C$, it follows from~\refclaim{cl:nosep} that $G$ has a $C$-path.
Let $P$ be a $C$-path chosen such that the number of vertices of $G-V(C\cup P)$ that belong to stable $C\cup P$-bridges is maximum.
We claim that $P$ is as desired.
To prove the claim we may assume for a contradiction that
there exists at least one unstable bridge.

A vertex $v$ of $P$ is \emph{straddled} \?{straddled} if it is an internal vertex of $P$ and there exists an unstable bridge with 
attachments in both components of $P - v$.  
We claim that there exists at least one straddled vertex in $P$.   Let $B$ be an unstable bridge.  If $B$ has at least three vertices, 
then it  has at least three attachments by~\refclaim{cl:nosep}, and therefore a middle attachment is straddled.  
Otherwise, if $B$ is has only two vertices, then its vertices are not adjacent in $P$ because $G$ is simple, and consequently, 
there exists a straddled vertex between the vertices of $B$.

Let $R$ be a maximal subpath of $P$ such that every internal vertex of $R$ is straddled.
Note that $R$ has length at least two.  As by~\refclaim{cl:nosep} the ends of $R$ do not form a vertex cut 
of size two separating the internal vertices of $R$ from $C$, we see there exists a $C\cup P$-bridge $B'$ with an attachment $x$ that is an 
internal vertex of $R$ and an attachment which is not contained in $R$.  If $B'$ were unstable, then it must straddle one of the ends of $R$, 
violating the maximality of $R$.  We conclude that $B'$ is stable.

The vertex $x$ is straddled by some unstable bridge $D$. Let $u,v$ be attachments of $D$ such that $u,x,v$ are distinct and
appear on $P$ in the order listed.  Let $P'$ be obtained from $P$ by replacing the subpath from $u$ to $v$ by a subpath of $D$ from $u$ to $v$.
It follows that every stable $C\cup P$-bridge is a subgraph of a stable $C\cup P'$-bridge, and the vertex $x$ belongs to a stable 
$C\cup P'$-bridge containing $B'$.
Thus the path \yy{$P'$} contradicts the choice of $P$. 
This proves~\refclaim{cl:stable}.
\bigskip

\bigskip

Let $P$ be a $C$-path in $G$ such that every $C\cup P$-bridge is stable.  


\myclaim{cl:jump}{No $C \cup P$-bridge has attachments in different components of $C-V(P)$.}

To prove~\refclaim{cl:jump} we note that if such a bridge existed, then it would include a  path $Q$  with ends in different components of $C-V(P)$.
But then the  paths $P$ and $Q$ form a $C$-cross, a contradiction, which proves~\refclaim{cl:jump}.
\bigskip

Let $C_1,C_2$ be the two cycles of $C\cup P$ other than $C$.
It follows from~\refclaim{cl:stable} and~\refclaim{cl:jump} that every $C\cup P$-bridge is either a $C_1$-bridge, or a $C_2$-bridge, and not both.
For $i=1,2$ let $G_i$ be the union of $C_i$ and all $C\cup P$-bridges of $G$ that are $C_i$-bridges.
Then $G_1\cup G_2=G$, $G_1\cap G_2=P$, and $|V(G_i)|+|E(G_i)|<|V(G)|+|E(G)|$ for $i=1,2$.

\myclaim{cl:sep}{For $i=1,2$ the graph $G_i$ has no elementary $C_i$-reduction.}

To prove~\refclaim{cl:sep} let $i\in\{1,,2\}$. If $G_i$ has an elementary $C_i$-reduction, then it
has a separation $(A, B)$ of order at most three with $C_i$ contained in $A$ and $B\setminus A\ne\emptyset$.
Then $(A \cup V(G_{3-i}), B)$ is a separation of $G$ contradicting~\refclaim{cl:nosep}. This proves~\refclaim{cl:sep}.
\bigskip

By induction and~\refclaim{cl:sep}, for $i=1,2$ the graph $G_i$ either has a $C_i$-cross, or can be drawn in the plane with $C_i$ bounding a face.
If the latter alternative holds for both $i=1$ and $i=2$, then the two drawings may be combined to produce a drawing of $G$ in the plane with
$C$ bounding a face, as desired. Thus we may assume without loss of generality that $G_1$ contains a $C_1$-cross $Q_1, Q_2$.  Let the ends of 
$Q_i$ be $s_i$ and $t_i$.  If $P$ contains at most two of the vertices $s_1, s_2, t_1, t_2$, we see that the cross $Q_1, Q_2$ readily extends to a $C$-cross in $G$ by possibly using subpaths of $P$, a contradiction.  


We claim that  we may assume that $\{s_1, s_2, t_1, t_2\} \nsubseteq V(P)$.
To prove this claim we may assume that $Q_1$ and $Q_2$ each have their both ends contained in $V(P)$.  
Since the $C\cup P$-bridge containing $Q_1$ is stable by~\refclaim{cl:stable}, it follows that $Q_1$ has an internal vertex, and
 there exists a path $R$ from an internal vertex of $Q_1$ or $Q_2$ to $V(C_1)\setminus V(P)$ 
and otherwise disjoint from $P \cup Q_1 \cup Q_2 \cup C_1$.  
We deduce that $R \cup Q_1 \cup Q_2$ contains a $C_1$-cross with at least one end not in $V(P)$, as desired.
This proves our claim that we may assume that  $\{s_1, s_2, t_1, t_2\} \nsubseteq V(P)$.

It now follows that $Q_1$ and $Q_2$ have a total of exactly three ends in $V(P)$.  Without loss of generality, assume that $s_1, s_2, t_1$ are contained in $V(P)$ and occur in that order when traversing $P$.  
Since the $C\cup P$-bridge containing $Q_1$ is stable by~\refclaim{cl:stable}, it follows that $Q_1$ has an internal vertex, and
 there exists a path $R$ from an internal vertex of $Q_1$  to $V(C_1)\setminus V(P)$ 
that is otherwise disjoint from $P \cup Q_1 \cup C_1$.
If $R$ is disjoint from $Q_2$, then $Q_1\cup Q_2\cup R$ includes a $C_1$-cross with exactly two ends in $P$, a case already handled.
Thus we may assume that $R$ has a subpath $S$ with one end in $Q_1-\{s_1,t_1\}$, the other end in $Q_2-s_2$, and otherwise disjoint from $Q_1\cup Q_2$.
Now  $\yy{S} \cup Q_1 \cup Q_2 \cup P$ is a tripod in $G$, contradicting~\refclaim{cl:tripod}.  This final contradiction completes the proof of the theorem.
\qed

The proof of Theorem~\ref{thm:crossreduct} is constructive and readily implies the existence of a polynomial time algorithm 
for the problem of Theorem~\ref{alg:crossreduct}.
However, it does not seem to achieve as good a bound on the running time as the algorithm of \zz{Kawarabayashi, Li and Reed~\cite{KawLiRee}}.


\baselineskip 11pt
\vfill
\noindent
This material is based upon work supported by the National Science Foundation.
Any opinions, findings, and conclusions or
recommendations expressed in this material are those of the authors and do
not necessarily reflect the views of the National Science Foundation.

\end{document}